\documentclass[amssymb,11pt]{amsart}
\usepackage{amscd,amssymb,euscript,amsthm}
\theoremstyle{plain}
\newtheorem{thm}{Theorem}[section] 
\newtheorem{cor}[thm]{Corollary}
\newtheorem{prop}[thm]{Proposition}
\newtheorem{lem}[thm]{Lemma}
\newtheorem{conj}[thm]{Conjecture}

\theoremstyle{definition}
\newtheorem{defn}[thm]{Definition}
\theoremstyle{remark}
\newtheorem{rem}[thm]{Remark}

\numberwithin{equation}{section}
\newcommand{\co}{\operatorname{conv}}

\def\<{\left<}
\def\>{\right>}
\def\cstar{$C^*$-algebra}
\begin{document}
\title[Hyperrigidity]{The noncommutative Choquet boundary II:\\
hyperrigidity}
\author{William Arveson}
%
%
\address{Department of Mathematics,
University of California, Berkeley, CA 94720}
\email{arveson@math.berkeley.edu}
%
\date{26 May, 2009}

\begin{abstract}  A (finite or countably infinite) set $G$ of generators 
of an abstract \cstar\ $A$ is called {\em hyperrigid} if  
for every faithful representation of $A$ 
on a Hilbert space $A\subseteq \mathcal B(H)$ and every sequence of unital completely positive 
linear maps $\phi_1, \phi_2,\dots$ from $\mathcal B(H)$ to itself,  
$$
\lim_{n\to\infty}\|\phi_n(g)-g\|=0, \forall g\in G \implies 
\lim_{n\to\infty}\|\phi_n(a)-a\|=0, \forall a\in A.  
$$
We show that one can determine whether a given 
set $G$ of generators is 
hyperrigid by examining the noncommutative Choquet boundary of the 
operator space spanned by $G\cup G^*$.  
We present a variety of concrete applications and discuss prospects for 
further development.  
\end{abstract}

\maketitle

\section{Introduction  }\label{S:in}

In a previous paper \cite{arvChoq} it was shown that every separable operator 
system has sufficiently many boundary 
representations, thereby providing a noncommutative counterpart of the 
function-theoretic fact that the closure of the Choquet boundary is the Silov boundary.  
Considering the central position of the latter in both potential theory and 
approximation theory, it is natural to expect corresponding 
applications of the noncommutative Choquet boundary  
to the theory of operator spaces.  In this paper we initiate 
a study of what might be called noncommutative approximation theory, focusing 
on the question: How 
does one determine whether a set of generators of a \cstar\ is hyperrigid?     

\begin{defn}\label{inDef1}
A finite or countably infinite set 
$G$ of generators of a \cstar\ $A$ is said to be {\em hyperrigid} if for every 
faithful representation $A\subseteq \mathcal B(H)$ of $A$ on a Hilbert space and 
every sequence  of unit-preserving 
completely positive (UCP) maps $\phi_n: \mathcal B(H)\to\mathcal B(H)$, $n=1,2,\dots$,  
\begin{equation}\label{inEq1}
\lim_{n\to\infty}\|\phi_n(g)-g\|=0, \ \forall g\in G \implies 
\lim_{n\to\infty}\|\phi_n(a)-a\|=0, \ \forall a\in A.  
\end{equation}
\end{defn}

We have lightened notation in this definition by identifying $A$ with its image $\pi(A)$ in a faithful 
nondegenerate representation $\pi: A\to \mathcal B(H)$ on a Hilbert space $H$.  
Significantly, hyperrigidity of a set $G$ of operators on a 
Hilbert space $H$ implies not only that (\ref{inEq1}) should 
hold for sequences of UCP maps $\phi_n$ defined on $\mathcal B(H)$, but also that the 
property should persist for every other faithful representation of $A$.  Note too 
that a set $G$ is hyperrigid iff the linear span of $G\cup G^*$ is hyperrigid, so that 
hyperrigidity is properly thought of as a property of self adjoint operator 
subspaces of a \cstar.  In 
principle, one could adjoin the identity to $G\cup G^*$ as well, but for many examples -- 
especially those involving sets of compact operators -- it is best not to adjoin 
the identity operator to $G$.  Hence 
we allow that a set $G$ of operators, its operator space and its generated \cstar \ may not 
contain a unit.  

The general characterization of hyperrigid generators given in Theorem 
\ref{mrThm1} provides the following criterion: {\em A separable operator system 
$S$ that generates a \cstar\ $A$ is hyperrigid iff every representation 
$\pi: A\to \mathcal B(H)$ on a separable Hilbert space $H$ has the {\em unique 
extension property} in the sense that the only unital completely positive (UCP) 
map $\phi: A\to\mathcal B(H)$ that satisfies $\phi\restriction_S=\pi\restriction_S$ 
is $\phi=\pi$ itself.  }

The simplest examples of hyperrigid 
generators $G$ are obtained by a direct application of this criterion.  
These examples are associated with 
``extremal" properties of the operators in $G$ which 
force the unique extension property (and therefore hyperrigidity) 
through a direct application of the Schwarz inequality and 
the Stinespring representation of UCP maps.  The  
following two results illustrate the point:  
They are proved in Section \ref{S:sa}.  

\begin{thm}\label{inThm3}
Let $X\in \mathcal B(H)$ be a self adjoint operator and let 
$\mathcal A$ be the \cstar\ generated by $X$.  Then $G=\{X, X^2\}$ is 
a hyperrigid generator for $\mathcal A$.  
\end{thm}

\begin{thm}\label{inThm2}
Let $V_1,\dots, V_n\in \mathcal B(H)$ be an arbitrary set of isometries that generates 
a \cstar\ $\mathcal A$.  Then $G=\{V_1,\dots, V_n, V_1V_1^*+\cdots+V_nV_n^*\}$ is 
a hyperrigid generator for $\mathcal A$.  
\end{thm}

\begin{rem}  Theorem \ref{inThm3}  can be viewed as a 
noncommutative strengthening of a classic approximation-theoretic result 
of Korovkin: see Remark \ref{inRem1} for further discussion. 
The referee has pointed out that Theorem \ref{inThm3} can be formulated in terms of 
the multiplicative domains of certain UCP maps, and 
after that reformulation, Lemma 3.1 of \cite{JOR} gives a norm estimate that 
leads to an alternate proof of Theorem \ref{inThm3} when applied 
to the function system $\{\mathbf 1, x, x^2\}\subseteq C[a,b]$, where 
$a$ and $b$ are appropriate bounds on the spectrum 
of the operator $X$. 

Finally, note that Theorem \ref{inThm2} implies that  for every $n\geq 2$, 
the standard set of generators $G=\{V_1,\dots, V_n\}$ of the Cuntz \cstar\ $\mathcal O_n$ 
is hyperrigid.  The referee has also pointed out a related result of Neshveyev and 
St{\o}rmer (Theorem 6.2.6 of \cite{NeshStor}), concerning generating sets 
of unitary operators.  
\end{rem}

On the other hand, we emphasize that most hyperrigid operator systems $S\subseteq C^*(S)$ do 
not share the conspicuous extremal properties 
associated with  Theorems \ref{inThm3} and \ref{inThm2}, and 
one cannot 
establish hyperrigidity of the more subtle examples by such direct methods.    
The purpose of this paper is to identify the obstruction to hyperrigidity {\em in general} 
in terms of the noncommutative Choquet boundary.  We conjecture that this is the only 
obstruction in Section \ref{S:ncb}.  While we are unable to 
establish the conjecture in general, we do prove it 
when $C^*(S)$ has countable spectrum, and that leads 
to a variety of hyperrigidity results with distinctly new features. 
We now describe 
two more subtle examples which are 
concrete special cases of more general results that are proved in Sections \ref{S:cs} through \ref{S:cr}.  

\vskip0.1in
\noindent
{\bf Positive linear maps of matrix algebras.}
Building on work of Chandler Davis \cite{DaJen}, 
Choi showed in \cite{ChoiSch} that for a unit-preserving 
positive linear map 
$\phi$ of unital \cstar s, the inequality 
\begin{equation}\label{crEq4}
f(\phi(A))\leq \phi(f(A))
\end{equation}
holds for every function $f:(a,b)\to\mathbb R$ that is {\em operator convex} 
in the sense of Bendat-Sherman \cite{bendSher} and every self adjoint 
operator $A$ with spectrum in $(a,b)$.  
Note that the spectrum of $\phi(A)$ is also contained 
in $(a,b)$, so that one can form both $f(A)$ and $f(\phi(A))$ by way of 
the functional calculus.  In \cite{petz}, Petz asked when {\em equality} can hold 
in (\ref{crEq4}), and showed that if 
$f:(a,b)\to \mathbb R$ is an operator convex function that is {\em not} of the form $f(x)= ax+b$ and 
equality holds in (\ref{crEq4}),  then 
the restriction of $\phi$ to the algebra 
of polynomials in 
$A$ is multiplicative.  

We want to broaden Petz' question in the following way.  Fix a real valued 
continuous function $f: [a,b]\to \mathbb R$ defined on a compact interval.  
We say that $f$ is {\em rigid} if 
for every every self adjoint operator $A$ in a unital \cstar\ $\mathcal A$ whose spectrum 
is contained in $[a,b]$ and  
every unit preserving positive linear map $\phi: \mathcal A\to\mathcal B$ into 
another unital \cstar,  one has 
$$
f(\phi(A))=\phi(f(A))\implies \phi(A^n)=\phi(A)^n, \quad n=1,2,\dots.  
$$
The following result -- a consequence of Theorem \ref{crThm1} below -- 
characterizes the rigid functions with respect 
to maps on matrix algebras, that is, the \cstar s $\mathcal A=\mathcal B(H)$ with finite 
dimensional $H$:    

\begin{thm}\label{inThm4} For every 
real-valued function $f\in C[a,b]$, the following are equivalent:
\begin{enumerate}
\item[(i)] For every unital positive linear map 
of matrix algebras $\phi: \mathcal M\to \mathcal N$ and every self adjoint operator $A\in \mathcal M$ 
having spectrum in $[a,b]$, 
$$
\phi(f(A))=f(\phi(A))\implies \phi(A^n)=\phi(A)^n, \quad \forall n=1,2,\dots.  
$$
\item[(ii)] $f$ is either strictly convex or strictly concave.  
\end{enumerate} 
\end{thm}

Recall that a real function $f\in C[a,b]$ is said to be {\em strictly convex} if for 
any two distinct points $x\neq y$ in $[a,b]$ 
and every $t\in (0,1)$, 
$$
f(t\cdot x+(1-t)\cdot y)<t \cdot f(x)+(1-t)\cdot f(y).  
$$
$f$ is said to be {\em strictly concave} when $-f$ is strictly convex.

\begin{rem}[Relation to Petz' theorem]
It follows from 
the characterization of \cite{bendSher} that an operator convex function $f$ that is 
{\em not} an affine function must be real-analytic with $f^{\prime\prime}>0$ 
throughout $(a,b)$.    
Since such functions are 
strictly convex, Theorem \ref{inThm4} implies Petz' result for maps on 
matrix algebras.  
Since  most continuous strictly convex functions are {\em not} operator convex, 
this is a significant extension of the result of \cite{petz}.  

It is natural to ask if Theorem \ref{inThm4} holds for unital positive 
linear maps of more general unital \cstar s; indeed, we will show in Section \ref{S:cr} 
that the implication 
(i)$\implies$(ii) holds in that generality.  While we conjecture that the 
opposite implication (ii)$\implies$(i) holds as well, 
that has not been proved (see Section \ref{S:cr} 
for further discussion).  
\end{rem}

\noindent
{\bf Hyperrigid generators of $\mathcal K$.}  We conclude with 
a fourth hyperrigidity result about a familiar - perhaps the most familiar - compact operator.

\begin{thm}\label{inThm1}  Consider the Volterra integration operator $V$ acting on the Hilbert 
space $H=L^2[0,1]$,  
$$
Vf(x)=\int_0^x f(t)\,dt, \qquad f\in L^2[0,1].     
$$
It is well-known that $V$ is irreducible, 
generating the \cstar\ $\mathcal K$ of all compact operators.  
This operator has the following additional properties: 
\begin{enumerate}
\item[(i)] $G=\{V,V^2\}$ is hyperrigid; and in particular, for every sequence of unital 
completely positive maps 
$\phi_n: \mathcal B(H)\to \mathcal B(H)$ for which 
$$
\lim_{n\to\infty}\|\phi_n(V)-V\|=\lim_{n\to\infty}\|\phi_n(V^2)-V^2\|=0, 
$$
one has 
$$
\lim_{n\to\infty}\|\phi_n(K)-K\|=0
$$
for every compact operator $K\in\mathcal B(H)$.  
\item[(ii)] The smaller generating set $G_0=\{V\}$ of $\mathcal K$ is not hyperrigid.   
\end{enumerate}
\end{thm}

While the hyperrigidity property (i) of $\{V, V^2\}$ 
formally resembles the 
hyperrigidity property of $\{X, X^2\}$ in 
Theorem \ref{inThm3}, the two settings are fundamentally different because 
$V$ is not a self adjoint operator.  Indeed, while 
Theorem 
\ref{inThm3} is more or less a direct consequence of the Schwarz inequality and 
Stinespring's theorem, the proof of Theorem \ref{inThm1} makes essential 
use of the noncommutative Choquet boundary (see  Corollary \ref{voCor1}, 
a consequence of the more general Theorem \ref{voThm1}).  

The paper is organized into three parts.  Part 1 is relatively short and 
contains the basic characterization 
hyperrigid operator systems.  Using that characterization, we discuss two of the simplest 
examples of hyperrigid generators and prove Theorems \ref{inThm3} and \ref{inThm2}.  

In order to deal with the more subtle aspects of hyperrigid generators 
it is necessary to bring in the 
noncommutative Choquet boundary, and Part 2 is devoted to those issues.  
We show how boundary representations are involved in the 
obstruction to hyperrigidity in Corollary \ref{ncbCor1}, and following that, 
we conjecture that  
this is the {\em only} obstruction in general.  
We are unable to prove the conjecture in general, but we do prove it for generators 
of \cstar s that have countable spectra (Theorem \ref{csThm1}).  
When the generated \cstar\ is not unital, there is an additional 
obstruction associated with the ``point at infinity" and we identify that obstruction 
in concrete operator-theoretic terms 
in Theorem \ref{gnThm1}.  In Section \ref{S:po} we introduce the noncommutative counterparts 
of peak points and show how one uses them to identify boundary representations 
for examples involving compact operators in Theorem \ref{poThm1}.  

When a set $G$ of operators generates a commutative \cstar\ $\cong C(X)$, it is possible 
to formulate a ``localized" version of Conjecture \ref{ncbCon1}.  Part \ref{P:local} is 
devoted to a discussion of this kind of localization, and in Theorem \ref{ueThm1} we 
prove an appropriate local version of Conjecture \ref{ncbCon1}.  

We work extensively with representations $\pi: A\to \mathcal B(H)$ of \cstar s $A$ 
on Hilbert spaces $H$ throughout this paper, 
and we require that all representations should be nondegenerate.  Thus, 
$H$ should be the closed linear span of the set of vectors 
$\{\pi(a)\xi: a\in A, \ \xi\in H\}$; and if 
$A$ has a unit $\mathbf 1$ then this entails $\pi(\mathbf 1)=\mathbf 1_H$.  

Finally, a word about notation.  When dealing with abstract \cstar s $A$, it is customary 
to refer to elements of $A$ with lower case letters $a\in A$, while when 
dealing with \cstar s of operators $\mathcal A\subseteq\mathcal B(H)$ it seems more 
appropriate to refer to operators with upper case letters $A\in\mathcal A$, as 
we have already done in the introduction.  
Of course, the two usages are inconsistent.  
But it seems punctilious to insist on referring to an 
operator on a Hilbert space $H$ with $a\in \mathcal B(H)$, and we 
revert at times (in Sections \ref{S:vo} and \ref{S:cr})
to more traditional operator-theoretic notation.  Hopefully, 
this will not cause problems for the reader.

\begin{rem}[Quantizing Korovkin's theorem]\label{inRem1}
When specialized appropriately, Theorem \ref{inThm3} provides a noncommutative strengthening of 
a classical theorem of approximation theory.  
To review that briefly, a seminal theorem of P. P. Korovkin 
\cite{kor1}, \cite{kor2} makes the following assertion: If a sequence 
of positive linear maps 
$\phi_1,\phi_2,\dots:C[0,1]\to C[0,1]$ has the property 
$$
\lim_{n\to\infty}\|\phi_n(f_k)-f_k\|=0,\qquad k=0,1,2, 
$$
for the three functions $f_0(x)=1$,  $f_1(x)=x$,  $f_2(x)=x^2$, then 
$$
\lim_n\|\phi_n(g)-g\|= 0, \qquad \forall\ g\in C[0,1].
$$  
Korovkin's theorem 
generated considerable activity among researchers in approximation theory, and 
far-reaching generalizations were discovered  during the 1960s, following the realization that 
the fundamental principle underlying it is that every point of the unit interval 
is a peak point for the $3$-dimensional function system $[1,x,x^2]\subseteq C[0,1]$.   
The generalizations 
make essential use of the Choquet boundary, and in one way or another, 
those we have seen use 
the fact that the real functions in $C(X)$ form a {\em lattice}.  
We will not summarize those developments here, but refer the reader to 
\cite{bauerBd}, \cite{bauDon}, \cite{phelps}, \cite{phelps2}, \cite{saskin}  and 
the survey \cite{donner}.

Theorem \ref{inThm3} strengthens    Korovkin's theorem in a nontrivial way.  
To see that in concrete terms, 
consider the multiplication operator $X$ on $L^2[0,1]$, 
$$
(X\xi)(t)=t\cdot\xi(t), \qquad t\in [0,1], \quad \xi\in L^2[0,1].
$$  
For every 
sequence of UCP maps $\phi_1, \phi_2,\dots:\mathcal B(L^2[0,1])\to\mathcal B(L^2[0,1])$ that satisfies  
\begin{equation}\label{inEq2}
\lim_{n\to\infty}\|\phi_n(X)-X\|=\lim_{n\to\infty}\|\phi_n(X^2)-X^2\|=0, 
\end{equation}
Theorem \ref{inThm3} implies that $\phi_n(Y)$ 
converges in norm to $Y$ for every multiplication operator $Y=M_f$ 
with $f\in C[0,1]$.  Of course, if each of the given maps $\phi_n$ 
leaves the commutative \cstar\ $\mathcal A=\{M_f: f\in C[0,1]\}$ invariant, then 
this would follow from Korovkin's  theorem.  However we do {\em not} assume that; 
indeed, the spaces $\phi_n(\mathcal A)$ need not commute with $X$ or 
with each other.    
If one attempts to use the methods of 
classical approximation theory to prove this operator-theoretic 
result, one finds that the argument breaks down precisely because a pair of self 
adjoint operators $A, B$ acting on a Hilbert space need not have a least upper bound 
or greatest lower bound, even when $AB=BA$.   
\end{rem}

 Finally, I thank Erling St{\o}rmer for helpful comments on a draft of 
this paper.  
 This is the second of a series of papers 
in which applications of the noncommutative Choquet boundary to the theory 
of operator spaces are developed.    

\part{Basic results}

\section{Characterization of hyperrigidity }\label{S:mr}

We now prove the basic characterization of hyperrigid operator systems.  

\begin{thm}\label{mrThm1}  For every separable operator system $S$ that generates 
a \cstar\ $A$, the following are equivalent:
\begin{enumerate}
\item[(i)] $S$ is hyperrigid.  
\item[(ii)]  For every nondegenerate representation $\pi: A\to \mathcal B(H)$ on a 
separable Hilbert space and every sequence $\phi_n:A\to \mathcal B(H)$ of UCP maps,  
$$
\lim_{n\to\infty}\|\phi_n(s)-\pi(s)\|=0\  \forall s\in S\implies 
\lim_{n\to\infty}\|\phi_n(a)-\pi(a)\|=0\ \forall a\in A.  
$$
\item[(iii)] For every nondegenerate representation $\pi: A\to \mathcal B(H)$ on a separable 
Hilbert space, $\pi\restriction_S$ has the unique extension property.  
\item[(iv)] For every unital \cstar \ $B$,  every unital homomorphism of \cstar s $\theta: A\to B$ 
and every UCP map $\phi: B\to B$, 
$$
\phi(x)=x\ \forall x\in \theta(S)\implies  \phi(x)=x \ \forall x\in \theta(A).   
$$
\end{enumerate}
\end{thm}

\begin{proof}Since the implication (ii)$\implies$(iii) is trivial, 
we prove (i)$\implies$(ii) and (iii)$\implies$(iv)$\implies$(i).  

(i)$\implies$(ii): Let $\pi: A\to \mathcal B(H)$ be a nondegenerate representation 
on a separable Hilbert space and let $\phi_n: A\to \mathcal B(H)$ be a 
sequence of UCP maps such that $\|\phi_n(s)-\pi(s)\|\to 0$ for all $s\in S$.  

Let 
$\sigma: A\to\mathcal B(K)$ be a faithful representation 
of $A$ on another separable space $K$.  Then $\sigma\oplus\pi: A\to\mathcal B(K\oplus H)$ 
is a faithful representation, so that each of the linear maps 
$\omega_n: (\sigma\oplus\pi)(A)\to\mathcal B(K\oplus H)$ 
$$
\omega_n: \sigma(a)\oplus\pi(a)\mapsto \sigma(a)\oplus\phi_n(a), \qquad a\in A, 
$$
is unit preserving and completely positive.  By the extension theorem 
of \cite{arvSubalgI} $\omega_n$ can be extended to a UCP map 
$\tilde\omega_n: \mathcal B(K\oplus H)\to\mathcal B(K\oplus H)$.  Since 
$\phi_n\restriction_S$ converges to $\pi\restriction_S$ point-norm, 
$\tilde\omega_n$ converges point-norm to the 
identity map on $(\sigma\oplus\pi)(S)$.  So by hypothesis (i), $\tilde\omega_n$ 
must converge point-norm to the identity map on $(\sigma\oplus\pi)(A)$.  We conclude that  
for every $a\in A$, 
\begin{align*}
\limsup_{n\to\infty}\|\phi_n(a)-\pi(a)\|&\leq 
\limsup_{n\to\infty}\|\sigma(a)\oplus\phi_n(a)-\sigma(a)\oplus\pi(a)\|
\\
&=\lim_{n\to\infty}\|\tilde\omega_n(\sigma(a)\oplus\pi(a))-\sigma(a)\oplus\pi(a)\|
=0, 
\end{align*}
hence $\phi_n$ converges point-norm to $\pi$ on $A$.

(iii)$\implies$(iv):  Let $\theta: A\to B$ be a unit preserving $*$-homomorphism 
of \cstar s, 
and let $\phi: B\to B$ be a UCP map that satisfies $\phi(\theta(s))=\theta(s)$, 
$s\in S$.  We have to show that 
\begin{equation}\label{mrEq2}
\phi(\theta(a))=\theta(a), \qquad a\in A.
\end{equation}

For that, let $B_0$ be the {\em separable} \cstar\ of $B$ generated by 
$$
\theta(A)\cup \phi(\theta(A))\cup \phi^2(\theta(A))\cup\cdots.  
$$
By its construction, $\phi(B_0)\subseteq B_0$.  Since $B_0$ is separable, 
it has a faithful representation on some separable Hilbert space $H$, and after 
making the obvious identification we may assume that $B_0\subseteq\mathcal B(H)$.  

By the extension theorem of \cite{arvSubalgI}, there is a UCP map 
$\tilde\phi: \mathcal B(H)\to \mathcal B(H)$ that restricts to $\phi$ 
on $B_0$, and in particular $\tilde\phi(\theta(s))=\theta(s)$ for $s\in S$.  
Since $a\in A\mapsto \theta(a)\in \mathcal B(H)$ is a representation on 
a separable Hilbert space, hypothesis (iii) implies that $\tilde\phi$ 
must fix $\theta(A)$ elementwise.  We conclude that 
$\phi(\theta(a))=\tilde\phi(\theta(a))=\theta(a)$, $a\in A$, and (\ref{mrEq2}) is proved.

(iv)$\implies$(i):  Suppose that $A\subseteq \mathcal B(H)$ is faithfully 
represented on some Hilbert space $H$, and $\phi_1, \phi_2, \dots: \mathcal B(H)\to\mathcal B(H)$ 
is a sequence of UCP maps  satisfying 
$\lim_n\|\phi_n(s)-s\|=0$ for all $s\in S$.  We have to prove:  
\begin{equation}\label{mrEq1}
\lim_{n\to\infty}\|\phi_n(a)-a\|=0,\qquad \forall\ a\in A.  
\end{equation}
To that end, write $B=\mathcal B(H)$, let 
$\ell^\infty(B)$ be the \cstar\ of all bounded sequences 
with components in $B$ and let $c_0(B)$ be the ideal of all sequences in $\ell^\infty(B)$ 
that converge to zero in norm.  

Consider the UCP map $\phi_0: \ell^\infty(B)\to \ell^\infty(B)$ defined by 
$$
\phi_0(b_1, b_2, b_3, \dots)=(\phi_1(b_1), \phi_2(b_2), \phi_3(b_3),\dots).   
$$
This map carries the ideal $c_0(B)$ into itself, hence it promotes to a UCP map 
of the quotient 
$\phi:\ell^\infty(B)/c_0(B)\to\ell^\infty(B)/c_0(B)$ 
by way of 
$$
\phi(x+c_0(B))=\phi_0(x)+c_0(B), \qquad x\in \ell^\infty(B).
$$ 

Now consider the natural embedding $\theta: A\to\ell^\infty(B)/c_0(B)$,  
$$
\theta(a)=(a,a,a,\dots)+c_0(B).  
$$  
By hypothesis, $\|\phi_n(s)-s\|\to 0$ as $n\to \infty$ for $s\in S$, and therefore 
$$
\phi(\theta(s))=(\phi_1(s), \phi_2(s), \dots)+c_0(B)=
(s,s,\dots)+c_0(B)=\theta(s).  
$$
Hence $\phi$ restricts to the identity map on $\theta(S)$.  

Applying hypothesis (iv) to the inclusions 
$$
\theta(S)\subseteq \theta(A)\subseteq \ell^\infty(B)/c_0(B)
$$ 
and the UCP map $\phi: \ell^\infty(B)/c_0(B)\to \ell^\infty(B)/c_0(B)$, 
we conclude that $\phi$ must fix every element of $\theta(A)$.  Since  
$\theta(a)=(a,a,\dots)+c_0(B)$ and 
$$
\phi(\theta(a))=(\phi_1(a),\phi_2(a),\dots)+c_0(B), 
$$
we must have $(\phi_1(a)-a, \phi_2(a)-a,\dots)\in c_0(B)$, and (\ref{mrEq1}) follows.
\end{proof}

It is significant that hyperrigidity is preserved under passage to quotients: 

\begin{cor}\label{mrCor2}  Let $S$ be a hyperrigid separable operator system 
with generated \cstar\ $A$, let $K$ be an ideal in $A$ and let 
$a\in A\mapsto \dot a\in A/K$ be the quotient map.  Then $\dot S$ is a hyperrigid operator 
system in $A/K$.  
\end{cor}

\begin{proof}
An immediate consequence of property (ii) of Theorem \ref{mrThm1}.  
\end{proof}

\section{Applications I: Two basic examples}\label{S:sa}

\begin{thm}\label{saThm1}  Let $x\in \mathcal B(H)$ be a self adjoint operator with at least 
$3$ points in its spectrum and let $A$ be the \cstar\ generated 
by $x$ and $\mathbf 1$. Then 
\begin{enumerate}
\item[(i)] $G=\{\mathbf 1, x,x^2\}$ is a hyperrigid generator for $A$, while  
\item[(ii)] $G_0=\{\mathbf 1, x\}$ is not a hyperrigid generator for $A$.  
\end{enumerate}
\end{thm}

\begin{proof} (i):  By Theorem \ref{mrThm1}, it suffices to show 
that every nondegenerate representation $\pi: C^*(x)\to \mathcal B(K)$ has 
the unique extension property.  To prove that, let $\phi: A\to \mathcal B(K)$ 
be a UCP map that satisfies $\phi(x)=\pi(x)$ and $\phi(x^2)=\pi(x^2)$.  We 
have to show that $\phi$ is multiplicative on $A$.  

For that, Stinespring's theorem implies that there is a Hilbert space 
$L$ containing $K$ and a representation $\sigma: A\to \mathcal B(L)$ such 
that $\phi(a)=P\sigma(a)\restriction_K$, $a\in A$, where $P\in\mathcal B(L)$ is the 
projection onto $K$.  We have 
\begin{align*}
P\sigma(x)(\mathbf 1-P)\sigma(x)P&=P\sigma(x^2)P-P\sigma(x)P\sigma(x)P=\phi(x^2)P-\phi(x)^2P
\\
&=\pi(x^2)P-\pi(x)^2P=0.   
\end{align*}
This implies that $|(\mathbf 1-P(\sigma(x)P)|^2=0$, hence 
$(\mathbf 1-P)\sigma(x)P=0$, i.e.,  $\sigma(x)$ leaves 
$H$ invariant.  Since $A$ is the norm-closed algebra generated by 
$\mathbf 1$ and $x$, it follows that $\sigma(A)$ leaves $H$ invariant, and 
consequently 
$\phi(a)=P\sigma(a)\restriction_K$ is a multiplicative linear map. 

(ii):   Choose points $\lambda_1<\lambda_2<\lambda_3$ in the spectrum $\Sigma$ of $x$.  Then 
$\lambda_2$ is a convex combination of $\lambda_1$ and $\lambda_3$.  For 
$k=1,2,3$, let $\rho_k$ be the state of $A$ defined by 
$$
\rho_k(f(x))=f(\lambda_k), \qquad f\in C(\Sigma).  
$$
Each $\rho_k$ is an irreducible representation of $A$, and by preceding remark, 
the restriction of $\rho_2$ to the function system $S={\rm{span}}\{\mathbf 1, x\}$
is a convex combination of $\rho_1\restriction_S$ and $\rho_2\restriction_S$.  
Since $\rho_1\neq \rho_3$,  $\rho_2\restriction_S$ fails to have the unique extension property, 
and Theorem 
\ref{mrThm1} implies that $S$ is not hyperrigid.  
\end{proof}

Note that the hypothesis on the cardinality of the spectrum of $x$ was not 
used in the proof of item (i) of Theorem \ref{saThm1}.  

\begin{rem}[Other hyperrigid generators]  Let $I=[a,b]$ be a compact real interval and 
let $f: I\to\mathbb R$ be a continuous function and let $A\in\mathcal B(H)$ be a self 
adjoint operator with spectrum in $[a,b]$.  One can ask: Is $\{\mathbf 1, A, f(A)\}$ 
a hyperrigid generator of $C^*(A)$?  Theorem 
\ref{saThm1} answers affirmatively for the particular function $f(t)=t^2$; but the proof 
of Theorem \ref{saThm1} is tailored to this particular function.  In general, 
there is a stringent constraint: {\em If the answer to the above question 
is yes then $f$ must be either strictly 
convex or strictly concave}.  This is a consequence of results 
of Section \ref{S:cr} (see Proposition \ref{crProp3}).

Conversely, if $f$ is strictly convex or strictly concave, then for every self adjoint operator 
$A$ with {\em discrete} spectrum in $I$, $\{\mathbf 1, A, f(A)\}$ is a hyperrigid generator. 
This can be established by making use of Proposition \ref{ncbProp1} at the appropriate 
place in the 
proof of Theorem \ref{crThm1} below.    
We believe that the same is true without the discrete spectrum hypothesis, but that 
depends on the validity of the commutative case of Conjecture \ref{ncbCon1} 
(see Remark  \ref{crRem1}).  
\end{rem}

We now discuss a class of highly noncommutative examples.  
Let $u_1,\dots, u_n$ be an arbitrary set of isometries that act on some Hilbert space.  
The ``defect operator" $D=u_1u_1^*+\cdots+u_nu_n^*$ is positive and  its norm 
satisfies $1\leq \|D\|\leq n$, with many possibilities for $D$ depending on how 
the $u_k$ are chosen.  
In this section we exhibit a hyperrigid generator for the \cstar\ 
generated by $u_1,\dots, u_n$, assuming nothing about the 
structure of the defect operator or relations that may exist 
between the various $u_k$.  

\begin{thm}\label{ciThm1}
Let $u_1,\dots, u_n$ be a set of isometries that generate a \cstar\ $A$ and 
let  
\begin{equation}\label{ciEq1}
G=\{ u_1, \dots, u_n, u_1u_1^*+\cdots+u_nu_n^*\}.  
\end{equation}
Then $G$ is a hyperrigid generator for $A$.  
\end{thm}

\begin{proof} Let $S$ be the operator system spanned by $G\cup G^*$ and 
the identity.  By item (iii) of Theorem \ref{mrThm1}, it suffices to show 
that for every nondegenerate representation $\pi$ of $A$, 
$\pi\restriction_S$ has the unique extension property.  

To prove that, fix a representation $\pi: A\to \mathcal B(H)$ 
and let $v_1,\dots, v_n$ be the isometries $v_k=\pi(u_k)$, $k=1,\dots, n$.  
Let $\phi: A\to \mathcal B(H)$ be a UCP map satisfying 
$\phi(u_k)=v_k$, $1\leq k\leq n$, and 
$\phi(u_1u_1^*+\cdots+u_nu_n^*)=v_1v_1^*+\cdots+v_nv_n^*$.  
We have to show that $\phi=\pi$.  

For that, we use Stinespring's 
theorem to express $\phi$ in the form  
$$
\phi(x)=V^*\sigma(x)V, \qquad x\in A, 
$$
where $\sigma$ is a representation of $A$ on a Hilbert space $K$,  
$V: H\to K$ is an isometry, and which 
is minimal in the sense that $\sigma(A)VH$ spans $K$.  

We claim first that $\sigma(u_k)V=Vv_k$, $1\leq k\leq n$.  
Indeed, for $k=1,\dots, n$ we have 
$$
V^*\sigma(u_k)^*VV^*\sigma(u_k)V=\phi(u_k)^*\phi(u_k)=u_k^*u_k=\mathbf 1_{H}, 
$$
hence $V^*\sigma(u_k)(\mathbf 1-VV^*)\sigma(u_k)V=0$, so 
that $\sigma(u_k)$ leaves $VH$ invariant.  The claim follows because 
$\sigma(u_k)V=VV^*\sigma(u_k)V=V\phi(u_k)=Vv_k$.  

Note next that since $\sum_k v_kv_k^*=\pi(\sum_ku_ku_k^*)=\phi(\sum_k u_ku_k^*)$, we have 
\begin{align*}
\sum_{k=1}^n\sigma(u_k)VV^*\sigma(u_k)^*&=\sum_{k=1}^nVv_kv_k^*V^*=V\phi(\sum_{k=1}^n u_ku_k^*)V
\\
&=VV^*\sum_{k=1}^n\sigma(u_ku_k^*)VV^*
\\
&=\sum_{k=1}^nVV^*\sigma(u_k)\sigma(u_k^*)VV^*
\end{align*}
and since $\sigma(u_k)V=VV^*\sigma(u_k)V$ for all $k$, subtracting the 
left side from the right leads to 
$$
\sum_{k=1}^n VV^*\sigma(u_k)(\mathbf 1_K-VV^*)\sigma(u_k)^*VV^*=0,   
$$
and hence $(\mathbf 1_K-VV^*)\sigma(u_k)^*VV^*=0$ for all $k$.  We conclude that 
$VH$ is invariant under both $\sigma(u_k)$ and $\sigma(u_k)^*$ for all 
$k$; and  since $A$ is generated by the $u_k$ it follows that $\sigma(A)VH\subseteq VH$.  
By minimality, we must have $VH=K$, which implies that $V$ is unitary and 
therefore $\phi(x)=V^{-1}\sigma(x)V$ is a representation.  Since 
$\phi$ agrees with $\pi$ on a generating set, the desired conclusion 
$\phi=\pi$ follows.  
\end{proof}

Since the Cuntz algebras $\mathcal O_n$ are generated by sets of isometries 
$u_1,\dots, u_n$ satisfying the single condition $u_1u_1^*+\cdots+u_nu_n^*=\mathbf 1$, 
we can discard the identity operator from the 
generating set $G$ of (\ref{ciEq1}) to conclude:

\begin{cor}\label{ciCor1}  
The set $G=\{u_1,\dots,u_n\}$ of generators of the Cuntz algebra $\mathcal O_n$ 
is hyperrigid.  
\end{cor}

\part{Role of the noncommutative Choquet boundary}

\section{Obstruction to hyperrigidity }\label{S:ncb}

An {\em operator system} is a self adjoint linear subspace of a unital 
\cstar\ $A$ that contains the unit of $A$, and the $C^*$-subalgebra of $A$ 
generated by $S$ is denoted $C^*(S)$.  Given a unital completely 
positive (UCP) map $\phi$ from an operator system $S$ to a unital \cstar\ 
$B$, we say that $\phi$ has the {\em unique extension property} if it has 
a unique UCP extension $\tilde\phi: C^*(S)\to B$, and moreover this extension 
is multiplicative $\tilde\phi(xy)=\tilde\phi(x)\tilde\phi(y)$, $x,y\in C^*(S)$.  
By a {\em boundary representation} for $S$ we mean 
an {\em irreducible} representation $\pi: C^*(S)\to \mathcal B(H)$ 
such that $\pi\restriction_S$ has the unique 
extension property.  There is a more intrinsic characterization of the 
unique extension property that we do not require here (see Proposition 2.4 
of \cite{arvChoq}).  
Much of the discussion to follow rests on a result 
 of \cite{arvChoq}, which we repeat here for reference: 

\begin{thm}\label{ncbThm1} Every separable operator system $S\subseteq C^*(S)$  
has sufficiently many boundary representations 
in the sense that for every $n\geq 1$ and every $n\times n$ matrix 
$(s_{ij})$ with components $s_{ij}\in S$, one has 
$$
\|(s_{ij})\|=\sup_{\pi}\|(\pi(s_{ij}))\|, 
$$
the supremum on the right taken over all boundary representations $\pi$ for $S$. 
\end{thm} 

Let $X$ be a compact metrizable space and let $S\subseteq C(X)$ be a function system, namely 
a linear subspace of $C(X)$ that is closed under complex conjugation and contains the 
constant functions.  There is no essential loss if one assumes that $S$ separates points 
of $X$.  Let $p$ be a point of $X$; by a {\em representing measure} for $p$ one means 
a (Borel) probability measure $\mu$ on $X$ satisfying 
$$
\int_X f(x)\,d\mu(x) = f(p), \qquad f\in S.  
$$
The set $K_p$ of all representing measures for $p$ is a weak$^*$-compact convex subset 
of the dual of $C(X)$, and it contains the point mass $\delta_p$ concentrated at $p$.  
If $K_p=\{\delta_p\}$, then $p$ is said to belong to the Choquet boundary of $X$ (relative 
to $S$), sometimes written $\partial_S(X)$.  It is not obvious that the Choquet boundary 
is nonempty; but it is always nonempty when $S$ separates points, and in fact its 
closure is the Silov boundary - the smallest closed set $K\subseteq X$ with the property that every 
function in $S$ achieves its maximum value on $K$ (see Proposition 6.4 of \cite{phelps2}).  
The following comments show that 
Theorem \ref{ncbThm1} generalizes this fact to noncommutative operator systems.

For every operator system $S\subseteq A=C^*(S)$, there is a largest (closed two 
sided) ideal $K\subseteq A$ such that the quotient map 
$a\in A\mapsto \dot a\in A/K$ is completely isometric on $S$.  The quotient 
\cstar\ $A/K$ is called the {\em $C^*$-envelope} of $S$.  The $C^*$-envelope of $S$ 
depends only on 
the internal structure of $S$ and not 
on the embedding of $S$ in its generated \cstar.  This ideal was 
introduced and shown to exist for a variety of examples in \cite{arvSubalgI}, 
where it was called the {\em Silov boundary ideal} since it is the noncommutative 
counterpart of the Silov boundary of a function system.  
The existence of the Silov boundary ideal in general was left open, and the 
issue was later settled affirmatively by Hamana 
\cite{hamInj1},  \cite{hamInj2} as a consequence of 
his work on injective envelopes.  
More recently, Dritschel and McCullough 
 \cite{MR2132691} gave  
a second proof of the existence of this ideal in general that is independent 
of the theory of injective envelopes.  During the past 
decade or so, the terminology for the ideal $K$ has been contracted to {\em Silov ideal}
for $S$.  
On the other hand, in the noncommutative context it seems more appropriate to refer to $K$ 
simply as the {\em boundary ideal} for $S$, as we shall do throughout this paper.  

It was shown in Theorem 2.2.3 of 
\cite{arvSubalgI} that for every operator system that has 
sufficiently many boundary representations (in the sense of Theorem \ref{ncbThm1}), 
the boundary ideal is the intersection of the kernels of 
all boundary representations.  
Note that the existence of sufficiently 
many boundary representations in general was left open in \cite{arvSubalgI} and 
\cite{arvSubalgII}, and was not addressed in Hamana's work on injectivity. 
Since Theorem \ref{ncbThm1} establishes 
that property for separable operator systems, it 
provides a third proof of the existence of the boundary ideal in such cases.    
Of course, this is the noncommutative 
counterpart of the 
fact that the closure of the Choquet boundary of a 
function system is the Silov boundary.

We deduce the following necessary conditions for hyperrigidity:  

\begin{cor}\label{ncbCor1}  Let $S$ by a separable operator system generating a \cstar\ $A$.  If $S$ is hyperrigid, 
then every irreducible representation of $A$ is a boundary representation for $S$.  
In particular, the boundary ideal of a hyperrigid operator system must be $\{0\}$.  
\end{cor}

\begin{proof}
The first assertion is an immediate consequence of condition (ii) of Theorem \ref{mrThm1}.  
The second follows from it, together with Theorem \ref{ncbThm1}, which implies
that the boundary ideal is the intersection of the kernels of all boundary representations 
for $S$.  
\end{proof}

We now conjecture that the obstructions described in Corollary \ref{ncbCor1} 
are the only obstructions to hyperrigidity.  Indeed, we will prove that conjecture 
for classes of examples in Section \ref{S:cs}:    

\begin{conj}\label{ncbCon1} If every 
irreducible representation of $A$ is a boundary representation for a separable operator 
system $S\subseteq A$, then $S$ 
is hyperrigid.  
\end{conj}

It is known that a direct sum of UCP maps with the unique extension property 
has the unique extension property (see \cite{MR2132691}).  For completeness, we conclude 
the section by proving that result in the form we require.  

\begin{prop}\label{ncbProp1}  Let $S\subseteq A=C^*(S)$ be an operator system, 
and for each $i$ in an index set $I$, 
let $\pi_i: A\to \mathcal B(H_i)$ be a 
representation such that $\pi_i\restriction_S$ has 
the unique extension property.  Then the direct sum of UCP maps 
$$
\oplus_{i\in I} \pi_i\restriction_S: S\to \mathcal B(\oplus_{i\in I}H_i)
$$
has the unique extension property.  
\end{prop}

\begin{proof}Let $\phi: A\to \mathcal B(\oplus_{i\in I}H_i)$ be an extension 
of $\pi$ to a UCP map from $A$ to $\mathcal B(\oplus_{i\in I}H_i)$, and  
for each $i\in I$, let $\phi_i: A\to \mathcal B(H_i)$ 
be the UCP map 
$$
\phi_i(a)=P_i\phi(a)\restriction_{H_i}, \qquad a\in A.  
$$
where $P_i$ is the projection on $H_i$.  
Since $\phi_i$ restricts to $\pi_i$ on $S$, the unique extension property 
of $\pi_i\restriction_S$ implies 
that $\phi_i(a)=\pi_i(a)$ for all $a\in A$, or equivalently, 
$P_i\phi(a)P_i=\pi(a)P_i$.  By the Schwarz inequality applied to $\phi$, 
\begin{align*}
P_i\phi(a)^*(\mathbf 1-P_i)\phi(a)P_i&=P_i\phi(a)^*\phi(a)P_i-P_i\phi(a)^*P_i\phi(a)P_i
\\
&\leq P_i\phi(a^*a)P_i-P_i\phi(a)^*P_i\phi(a)P_i
\\
&=\pi(a^*a)P_i-\pi(a)^*\pi(a)P_i=0.  
\end{align*}
Hence $|(\mathbf 1-P_i)\phi(a)P_i|^2=0$, and 
it follows that $P_i$ commutes with the self adjoint family of operators $\phi(A)$.  
So for every $a\in A$ we have 
$$
\phi(a)=\sum_{i\in I}\phi(a)P_i=\sum_{i\in I}P_i\phi(a)P_i=\sum_{i\in I}\pi(a)P_i=\pi(a)
$$
as asserted.  
\end{proof}

\section{Countable spectrum }\label{S:cs}

Let $A$ be a separable \cstar.  By the {\em spectrum} of $A$ we mean the set 
$\hat A$ of unitary equivalence classes of irreducible representations of $A$.  In 
general, $\hat A$ carries a natural Borel structure that separates points of $\hat A$, 
and it is well-known that 
$A$ is type I iff the Borel structure of $\hat A$ is countably separated.  
In this section we prove Conjecture \ref{ncbCon1} for operator systems $S$ whose generated 
\cstar\ has countable spectrum.  This class of \cstar s 
includes those generated by sets of compact operators 
(and the identity) 
as well as many others.  It is closed under most of the natural 
ways of forming new \cstar s from given ones (countable direct sums, quotients, ideals, 
extensions, crossed products with compact Lie groups), 
but of course it fails to contain most commutative \cstar s.

\begin{thm}\label{csThm1}
Let $S$ be a separable operator system whose generated \cstar \ $A$ has 
countable spectrum, such that every irreducible representation 
of $A$ is a boundary representation for $S$.  Then $S$ is hyperrigid.  
\end{thm}

\begin{proof}
By item (iii) of Theorem \ref{mrThm1}, it suffices to show that for every representation 
$\pi:A\to \mathcal B(H)$ of $A$ on a separable Hilbert space, the UCP map 
$\pi\restriction_S$ has the unique extension property.   Since the spectrum 
of $A$ is countable, $A$ is 
a type I \cstar, hence $\pi$ decomposes uniquely into a direct integral of 
mutually disjoint 
type I factor representations.  Using countability of $\hat A$ again, 
the direct integral must in fact be a countable direct sum.  Hence 
$\pi$ can be decomposed into a direct sum of 
subrepresentations $\pi_n: A\to \mathcal B(H_n)$
\begin{equation}\label{csEq1}
H=H_1\oplus H_2\oplus\cdots, \qquad \pi=\pi_1\oplus \pi_2\oplus\cdots 
\end{equation}
with the property that each $\pi_n$ is unitarily equivalent to 
a finite or countable direct sum of 
copies of a single irreducible 
representation $\sigma_n: A\to \mathcal B(K_n)$.   

By hypothesis, each UCP map $\sigma_n\restriction_S$ has the unique extension property.  
Hence the above decomposition expresses $\pi\restriction_S$ as a (double)
direct sum of UCP maps with the unique extension property.   
By Proposition \ref{ncbProp1}, it follows that $\pi\restriction_S$ has the unique extension property.  
\end{proof}

\section{Generators of nonunital \cstar s }\label{S:gn}

In this section we discuss sets $G$ of operators that generate a nonunital \cstar\ $A$ -- for example, 
sets of compact operators on an infinite dimensional Hilbert space.  
One can adjoin the identity operator to obtain a unital 
\cstar \ $\tilde A=A+\mathbb C\cdot\mathbf 1$, at the cost of introducing 
an additional one dimensional irreducible representation $\pi_\infty: \tilde A\to \mathbb C$ 
that represents ``evaluation at $\infty$"
\begin{equation}\label{gnEq1}
\pi_\infty(a+\lambda\cdot\mathbf 1)=\lambda, \qquad a\in A, \quad \lambda\in\mathbb C.  
\end{equation}
It is a fact that 
$\pi_\infty$ may or may not be a boundary representation for the operator system  $\tilde S$ 
spanned by $G\cup G^*\cup\{\mathbf 1\}$; and when it is not a boundary representation, 
$G$ cannot be hyperrigid.  
The purpose of this section is to identify this obstruction to hyperrigidity 
in concrete operator-theoretic terms.  
We will show that 
$\pi_\infty$ is a boundary representation for $\tilde S$ iff the original (nonunital) space $S$ 
spanned by $G\cup G^*$ 
``almost contains" strictly positive operators.  

A self adjoint operator $x\in A$ is said to be {\em almost dominated} by $S$ if there is 
a sequence of self adjoint operators $s_n\in S$ such that
$$
s_n+\frac{1}{n}\cdot\mathbf 1\geq x, \qquad n=1,2,\dots.     
$$
A more familiar notion is strict positivity:  A positive operator $p\in A$ 
is called {\em strictly positive} if for every 
positive linear functional $\phi\in A^\prime$, 
$$
\phi(p)=0\implies \phi=0.  
$$  
It is well-known that separable \cstar s contain many strictly positive operators; for 
example, if $e_1\leq e_2\leq \cdots$ is a countable approximate unit for $A$, then 
for every sequence of positive numbers $c_1, c_2, \dots$ with finite sum, 
$$
p=c_1\cdot e_1+c_2\cdot e_2+\cdots 
$$
is a strictly positive operator in $A$.  

\begin{thm}\label{gnThm1} Let $S$ be a self adjoint operator space that generates 
a nonunital \cstar\ $A$, let $\tilde A=A+\mathbb C\cdot \mathbf 1$, 
$\tilde S=S+\mathbb C\cdot\mathbf 1$, and let $\pi_\infty: \tilde A\to \mathbb C$ 
be the representation at $\infty$.  The following are equivalent.   
\begin{enumerate}
\item[(i)] $\pi_\infty$ is a boundary representation for $\tilde S$.  
\item[(ii)] $A$ contains a strictly positive operator  that is almost 
dominated by $S$.  
\item[(iii)] Every self adjoint operator $x\in A$ is almost dominated 
by $S$.  
\end{enumerate}
\end{thm}

Our proof of Theorem \ref{gnThm1} requires an operator-algebraic variation of a classic 
minimax principle -  a consequence of Krein's extension theorem 
for positive linear functionals.  While the result is known in one form or another 
to specialists, we lack a specific reference and include a proof 
for completeness.  
Let $S$ be an operator system and let $B$ be the (unital) \cstar\ generated by $S$.  A 
{\em state} of $S$ is a positive linear functional $\phi$ on $S$ such that 
$\phi(\mathbf 1)=1$.  Krein's extension theorem implies that 
every state of $S$ can be extended to a state of $B$, and 
we write $E_\phi$ for the weak$^*$-compact convex set of all extensions of $\phi$ 
to a state of $B$.  

\begin{prop}\label{gnProp1} Let $S$ be an operator system that generates a \cstar\ $B$.  
For every state $\phi$ of $S$ and every self-adjoint operator $x\in B$, 
\begin{align*}
\sup\{\phi(s): s=s^*\in S,\ s\leq x\}&=\min\{\rho(x): \rho\in E_\phi\}, \\
\inf\{\phi(s): s=s^*\in S,\ s\geq x\}&=\max\{\rho(x): \rho\in E_\phi\} . 
\end{align*}
\end{prop}

\begin{proof}[Proof of Proposition \ref{gnProp1}]
We prove the first formula; the second one follows from it by replacing 
$x$ with $-x$.  If $\rho\in E_\phi$ and $s=s^*\leq x$, then 
$$
\phi(s)=\rho(s)\leq \rho(x)
$$
and one obtains $\leq $ after taking the sup over $s$ and the inf over $\rho$.  

For the inequality $\geq$, let $L$ be the left hand side.  We claim that there 
is a $\rho\in E_\phi$ with $L= \rho(x)$.  For the proof, we may assume that $x\notin S$, and 
consider the linear functional defined on the operator system 
$S+\mathbb C\cdot x$ by 
$$
\hat\phi(s+\lambda x)=\phi(s)+\lambda  L, \qquad s\in S,\quad \lambda\in \mathbb C.  
$$
We claim that $\hat\phi$ is a state of $S+\mathbb C\cdot x$.  
Since $\hat\phi(\mathbf 1)=1$, after rescaling, 
this reduces to checking $s+ x\geq 0\implies s+ L\geq 0$ and 
$s-x\geq 0\implies \phi(s)-L\geq 0$, where in both 
cases $s$ is a self-adjoint element of $S$.  

If
$s+x\geq 0$, then $x\geq -s$ so that $-\phi(s)=\phi(-s)\leq L$, hence $\phi(s)+L\geq 0$.  
If $s-x\geq 0$, then for every $t=t^*\in S$ satisfying $t\leq x\leq s$ we 
have $t\leq s$, hence $\phi(t)\leq \phi(s)$ and therefore $L\leq \phi(s)$ by the 
arbitrariness of $t$.  The desired inequality $\phi(s)-L\geq 0$ follows.  

By Krein's extension theorem, $\hat\phi$ can be extended to a state $\rho$ of 
$B$, and such an extension satisfies $\rho\in E_\phi$ and $L=\rho(x)$.  
\end{proof}

\begin{proof}[Proof of Theorem \ref{gnThm1}] Since $A$ must contain 
strictly positive elements, the implication (iii)$\implies$(ii) is trivial.  We 
prove (i)$\implies$(iii) and (ii)$\implies$(i).  

(i)$\implies$(iii):  Let $x$ be a self adjoint element of $A$.  Applying the 
second formula of Proposition 
\ref{gnProp1} to the operator system $\tilde S$ and its state $\phi=\pi_\infty\restriction_{\tilde S}$ 
and noting that $E_\phi=\{\pi_\infty\}$ by hypothesis (i), 
we find that 
$$
\inf\{\lambda\in \mathbb R: \exists s=s^*\in S, \ s+\lambda\cdot\mathbf 1\geq x\}=\pi_\infty(x)=0.  
$$
It follows that there is a sequence $s_n=s_n^*\in S$ such that $s_n+\frac{1}{n}\cdot\mathbf 1\geq x$, 
hence $x$ is almost dominated by $S$.  

(ii)$\implies$(i):  Assuming (ii), let $\rho$ be a state of $\tilde A$ that satisfies 
$\rho\restriction_{\tilde S}=\pi_\infty\restriction_{\tilde S}$.  We have 
to show that $\rho=\pi_\infty$.   
To that end, choose a strictly positive element $p\in A$ 
 that is almost dominated by $S$, and consider the positive linear functional $\sigma\in A^\prime$ 
defined by $\sigma=\rho\restriction_A$.  By the hypothesis on $p$ there is a sequence 
$s_n=s_n^*\in S$ such that $s_n+\frac{1}{n}\cdot\mathbf 1\geq p$ for $n=1,2,\dots$.  
Applying $\rho$ to this inequality and using $\rho(s_n)=\pi_\infty(s_n)=0$, we conclude  
that 
$$
\frac{1}{n}\geq \sigma(p)\geq 0,  \qquad n=1,2,\dots,
$$ 
hence $\sigma(p) = 0$.  
It follows that $\sigma=0$ by strict positivity of $p$, which implies the desired 
conclusion $\rho=\pi_\infty$.  
\end{proof}

The following sufficient condition is easy to check for many examples.   

\begin{cor}\label{gnCor1}
Let $S\subseteq A$ be as in Theorem \ref{gnThm1}.  
If $S$ contains a strictly positive operator of $A$ 
then $\pi_\infty: \tilde A\to \mathbb C$ is a boundary representation for $\tilde S$.  
\end{cor}

\begin{proof}
If  $S$ itself contains a strictly positive operator $p$, then condition (ii) 
of Theorem \ref{gnThm1} is satisfied.  
\end{proof}

\section{Peaking representations  }\label{S:po}

Given an exact sequence of \cstar s 
$$
0\longrightarrow K\longrightarrow A\longrightarrow B\longrightarrow 0
$$
in which $a\in A\mapsto \dot a\in A/K= B$ is  
the natural quotient map, recall that every nondegenerate representation 
$$
\pi: A\to \mathcal B(H)
$$ 
of $A$ decomposes uniquely into a central direct sum 
of representations 
$$
\pi=\pi_K\oplus \pi_B
$$
where $\pi_K$ is the unique extension to $A$ of a nondegenerate representation 
of the ideal $K$, and where $\pi_B$ is a nondegenerate 
representation of $A$ that annihilates 
$K$.  When $\pi=\pi_K$ we say that $\pi$ {\em lives on} $K$.  In 
an obvious sense, 
the spectrum of $A$ decomposes into a disjoint union 
\begin{equation}\label{poEq1}
\hat A = \hat K \cup \hat B.  
\end{equation}

Now let $S\subseteq \mathcal B(H)$ be a concrete operator system that generates a 
\cstar \ $A$.  In general, the set $K$ of all compact operators in $A$ is a closed 
two-sided ideal.  In this section 
we address the problem of identifying the points of 
$\hat K$ that correspond to boundary representations for $S$ 
in cases where $K\neq \{0\}$, 
and we show 
how one can identify the boundary representations of $\hat K$ as 
noncommutative counterparts of peak points of function systems.

\begin{defn}  Let $S$ be a separable operator system that generates a \cstar\ $A$.  
An irreducible representation $\pi: A\to\mathcal B(H)$ is said to be {\em peaking} 
for $S$ if there is an $n\geq 1$ and an $n\times n$ matrix $(s_{ij})$ over 
$S$ such that 
\begin{equation}\label{poEq2}
\|(\pi(s_{ij}))\|>\|(\sigma(s_{ij}))\|
\end{equation}
for every irreducible representation $\sigma$ inequivalent to $\pi$, 
written  $\sigma\nsim \pi$.
$\pi$ is said to be {\em strongly peaking} if there is an $n\geq 1$ and an 
$n\times n$ matrix $(s_{ij})$ over $S$ such that 
\begin{equation}\label{poEq3}
\|(\pi(s_{ij}))\|>\sup_{\sigma\nsim\pi}\|(\sigma(x_{ij}))\|.   
\end{equation}
\end{defn}

An $n\times n$ matrix $(s_{ij})$ satisfying (\ref{poEq2}) (resp.\,(\ref{poEq3}) )is 
called a {\em peaking operator} (resp. {\em strong peaking operator}) for $\pi$.  Strongly 
peaking irreducible representations correspond to isolated points of $\hat A$, 
and they arise naturally when compact operators are present - such as in the setting  
of Theorem \ref{poThm1} below.  We shall have nothing more to say about 
peaking representations that are not strongly peaking in this paper.

The following characterization 
of boundary representations generalizes the Boundary Theorem of \cite{arvSubalgII}, 
and provides the basis for more concrete results on hyperrigid 
sets of compact operators such as Corollary \ref{poCor1} and Theorem \ref{voThm1}.

\begin{thm}\label{poThm1} Let $S\subseteq \mathcal B(H)$ be a separable concrete operator 
system and let $A$ be the \cstar\ generated by $S$.  Let $K$ be the ideal of 
all compact operators in $A$, assume that $K\neq \{0\}$, and  
let $\hat K$ be the set 
of unitary equivalence classes of irreducible representations of $A$ that live on $K$.  

Then $\hat K$ contains boundary representations for $S$ iff the quotient map 
$$
x\in A\mapsto \dot x\in A/K
$$ 
is {\em not} completely isometric on $S$.  Assuming that is the case, then 
among the irreducible representations of $\hat K$, 
the boundary representations for $S$ 
are precisely the strongly peaking 
ones.  
\end{thm}

\begin{proof}  If $\hat K$ contains no boundary representations, then because 
of the dichotomy 
(\ref{poEq1}), every 
boundary representation must annihilates $K$, and consequently it  
factors through the quotient map $a\in A\mapsto \dot a\in A/K$.  By Theorem \ref{ncbThm1}, 
there are sufficiently many boundary representations $\pi_i$, $i\in I$, for $S$ so that 
$$
\|(\dot s_{ij})\|\leq \|(s_{ij})\|=\sup_{i\in I}\|(\pi_i(s_{ij}))\|\leq \|(\dot s_{ij})\|
$$
for every $n\times n$ matrix $(s_{ij})$ over $S$ and every $n\geq 1$.  Hence 
the quotient map is completely isometric on $S$.  Conversely, if the quotient map 
is completely isometric on $S$, then we claim that no $\pi\in \hat K$ can be a 
boundary representation.  Indeed, for every irreducible representation 
$\pi: A\to \mathcal B(H_\pi)$ that lives in $K$, the hypothesis implies that the map 
$$
\dot s\in \dot S\subseteq A/K\mapsto \pi(s)
$$
is completely positive, and hence can be extended to a completely 
positive linear map $\phi: A/K\to \mathcal B(H_\pi)$.  The 
map $a\in A\mapsto \phi(\dot a)$ is therefore a completely positive linear map 
that restricts to $\pi$ on $S$, and which annihilates $K$.  This map 
differs from $\pi$ because 
$\pi$ lives in $K$, hence  
$\pi$ does not have the unique extension property.

Turning now to the proof of the last sentence, enumerate the distinct elements 
of $\hat K$ as $\{\pi_1, \pi_2,\dots\}$, and view each $\pi_k$ as an irreducible 
subrepresentation of the identity representation of $A$, so that 
$\pi_k(a)=a\restriction_{H_k}$, $a\in A$, 
where $H_1, H_2, \dots\subseteq H$ are mutually orthogonal 
reducing subspaces for $A$.  

Assuming first that $\pi_1$, say, is a boundary representation for $S$, we claim 
that $\pi_1$ is strongly peaking for $S$.  Indeed, if $\pi_1$ were not strongly peaking, then 
for every $n\geq 1$ and every $n\times n$ matrix $(s_{ij})$ over $S$ we would have 
$$
\|(\pi_1(s_{ij})\|\leq \sup_\sigma\max(\|(\sigma(s_{ij}))\|, \|(\pi_2(s_{ij}))\|, \|(\pi_3(s_{ij}))\|,\dots).  
$$
where $\sigma$ ranges over all irreducible representations of $A$ that annihilate $K$.  
Let $\rho: A/K\to \mathcal B(L)$ be a faithful representation of $A/K$ and consider the 
representation $\tilde\rho$ of $A$ defined by 
$$
\tilde\rho(a)=\rho(\dot a)\oplus \pi_2(a)\oplus\pi_3(a)\oplus\cdots.  
$$ 
 The preceding 
inequalities imply that the map 
$$
\tilde\rho(s) \mapsto \pi_1(s), \qquad s\in S, 
$$
is completely contractive.  Since it is also unit-preserving, it must be completely 
positive, and hence by the extension theorem of \cite{arvSubalgI} there is a UCP 
map $\phi: \tilde\rho(A)\to \mathcal B(H_{\pi_1})$ such that 
$$
\phi(\tilde\rho(s))=\pi_1(s),\qquad s\in S.    
$$
Since the UCP map $\phi\circ\tilde\rho: A\to \mathcal B(H_{\pi_1})$ extends 
$\pi\restriction_S$ and $\pi_1$ is assumed to be a boundary representation for $S$, 
it follows that $\phi\circ\tilde\rho=\pi$ on $A$, and in particular, 
$$
\phi(\tilde\rho(k))=\pi_1(k),\qquad k\in K.  
$$
Noting that for $k\in K$, 
$$
\tilde\rho(k)=0\oplus\pi_2(k)\oplus\pi_3(k)\oplus\cdots, 
$$
it follows that the map $\pi_2(k)\oplus\pi_3(k)\oplus\cdots\mapsto \pi_1(k)$ is 
completely contractive, or equivalently, that the map 
$$
k\restriction_{H_2\oplus H_3\oplus\cdots}\mapsto \pi_1(k),\qquad k\in K, 
$$
defines an irreducible representation of the \cstar\ $K_0=K\restriction_{H_2\oplus H_3\oplus\cdots}$.  
Since $K_0$ is a \cstar\ of compact operators, $\pi_1$ must be unitarily equivalent to 
one of the irreducible subrepresentations of the identity representation 
of $K_0$, namely $\pi_2, \pi_3,\dots$, say 
$\pi_1\sim\pi_r$, for some $r\geq 2$.  It follows that $\pi_1$ is equivalent to $\pi_r$, 
and we have arrived at a contradiction.  Hence $\pi_1$ must have been strongly peaking for $S$.  

Conversely, assume that one of the elements of $\hat K$, say $\pi_1$, is strongly peaking.  Let 
$\{\sigma_i: i\in I\}$ be a complete set of mutually inequivalent boundary representations 
for $S$.  We claim that $\pi_1$ is equivalent to some $\sigma_i$, and is therefore a 
boundary representation.  Indeed, if that were not the case, 
then by definition of strong peaking representation (\ref{poEq3}), there would be an $n\geq 1$ and 
an $n\times n$ matrix $(s_{ij})$ over $S$ such that 
\begin{equation}\label{poEq4}
\|(\pi_1(s_{ij}))\|>\sup_{i\in I}\|(\sigma_i(s_{ij}))\|.  
\end{equation}
On the other hand, since the list $\{\sigma_i: i\in I\}$ contains all boundary representations up 
to equivalence, 
Theorem \ref{ncbThm1} implies that the right side of (\ref{poEq4}) is $\|(s_{ij})\|$.  
We conclude that $\|(\pi_1(s_{ij}))\|>\|(s_{ij})\|$, and hence the 
completely bounded norm of $\pi_1\restriction_S$ is $>1$.  But representations are completely 
contractive, hence the assumption that $\pi_1\nsim \sigma_i$ for all 
$i\in I$ was false.  
\end{proof}

The following result provides concrete criteria for checking hyperrigidity 
for generators of \cstar s of compact operators.  See Section \ref{S:vo} for 
specific examples of how one makes use of it.

\begin{cor}\label{poCor1}
Let $G\subseteq \mathcal B(H)$ be a finite or countably infinite set of compact operators 
on an infinite dimensional Hilbert space, 
let $S$ be the linear span of $G\cup G^*$ 
and let $A$ be the \cstar\ generated by $G$.  Then $G$ is hyperrigid iff 
\begin{enumerate}
\item[(a)]
Every irreducible subrepresentation of the identity representation of $A$ is 
strongly peaking for the operator system $S+\mathbb C\cdot\mathbf 1$, and 
\item[(b)]  $S$ almost dominates 
some strictly positive operator in $A$.  
\end{enumerate}
\end{cor}

\begin{proof}  Let $\tilde A=A+\mathbb C\cdot\mathbf 1$ be the unitalization of 
$A$ and let 
$\tilde S$ be the operator system spanned by $G\cup G^*$ and the identity 
operator.  The irreducible representations of $\tilde A$ are the 
irreducible subrepresentations $\pi_1, \pi_2,\dots$ 
of the identity representation of $\tilde A$, 
together with the one dimensional representation $\pi_\infty(a+\lambda\cdot\mathbf 1)=\lambda$, 
$a\in A$, $\lambda\in \mathbb C$.  By Theorem \ref{poThm1}, (a) is equivalent to the assertion 
that every irreducible subrepresentation of the 
identity representation is a boundary 
representation for $\tilde S$, and by Theorem \ref{gnThm1}, 
(b) is equivalent to the assertion that $\pi_\infty$ 
is a boundary representation for $\tilde S$.  Since the spectrum of $\tilde A$ is countable, 
Corollary \ref{ncbCor1} and Theorem 
\ref{csThm1} show that these assertions are equivalent to the hyperrigidity 
of $G$.  
\end{proof}

\section{Applications II: Volterra type operators }\label{S:vo}

In this section we identify a broad class of irreducible compact operators 
that includes the Volterra integration operator on $L^2[0,1]$, we 
show that for such operators $V$, $\mathcal G=\{V, V^2\}$ is a hyperrigid generator 
for the \cstar\ of compact operators, 
but that the smaller generator $\mathcal G_0=\{V\}$ is not hyperrigid.

By standard spectral theory, 
every self-adjoint operator $B$ decomposes uniquely into a difference 
$B=B_+-B_-$, where $B_\pm\geq 0$ and $B_+B_-=0$.  
A self-adjoint operator $B\in\mathcal B(H)$ is said to be {\em essential} if its 
positive and negative parts $B_+$ and $B_-$ both have infinite rank.  A straightforward 
argument shows that if $B$ is essential and $F$ is a self-adjoint finite rank 
operator, then $B+F$ is also essential.  

\begin{thm}\label{voThm1}  Let $V\in\mathcal B(H)$ be an irreducible compact operator 
with cartesian decomposition $V=A+iB$, where $A$ is a finite rank positive 
operator and $B$ is essential with $\ker B=\{0\}$.  Then 
\begin{enumerate}
\item[(i)] $\mathcal G=\{V, V^2\}$ is a hyperrigid generator for the \cstar\ $\mathcal K$
of compact operators.  In 
particular, for every sequence of unital completely positive maps  
$\phi_n:\mathcal B(H)\to\mathcal B(H)$ that satisfies  
$$
\lim_{n\to\infty}\|\phi_n(V)-V\|=\lim_{n\to\infty}\|\phi_n(V^2)-V^2\|= 0
$$
one has 
$$
\lim_{n\to\infty}\|\phi_n(K)-K\|=0 
$$
for every compact operator $K$.  
\item[(ii)] The subset 
$\mathcal G_0=\{V\}$ is not a hyperrigid generator.  
\end{enumerate}
\end{thm}

\begin{proof}  Note first that $V$ must generate the full \cstar\ $\mathcal K$ of compact 
operators, since $\mathcal K$ contains no proper irreducible $C^*$-subalgebras.  Let 
$\mathcal S$ be the linear span of $V, V^*, V^2, V^{2*}$ and let 
$\tilde{\mathcal S}=\mathcal S+\mathbb C\cdot \mathbf 1$.  
Then $\tilde{\mathcal S}$ is an operator system generating the \cstar\ 
$\mathcal K+\mathbb C\cdot\mathbf 1$, whose irreducible representations are $\pi_\infty$ 
and, up to equivalence, the identity representation.

(i):  The cartesian decomposition of $V^2=(A+iB)(A+iB)$ is 
$$
V^2=(A^2-B^2)+i(AB+BA).
$$
  Since $A$ is a positive finite rank operator, we can find a $c>0$ so that 
$A^2\leq c\cdot A$, hence 
$$
-c\cdot A+(A^2-B^2)\leq -B^2 < 0
$$
is a strictly negative operator in $\mathcal S$.  Corollary \ref{gnCor1}   
implies that $\pi_\infty$ is a boundary representation for 
$\tilde{\mathcal S}$.  The other irreducible 
representation of $C^*(\tilde{\mathcal S})$ is 
equivalent to the identity representation, and obviously $V$ is itself a peaking 
operator for the identity representation restricted to $\mathcal S$.  
Theorem \ref{poThm1} implies that the identity representation is a boundary 
representation for $\tilde{\mathcal S}$, so by Corollary \ref{poCor1}, $\mathcal G=\{V, V^2\}$ is a hyperrigid 
generator for $\mathcal K$.  

(ii):  Consider the operator space $S_0={\rm span}\{A, B\}$ and let $Q$ be the projection on $AH^\perp$.  Since 
$Q$ is of finite codimension, $QBQ$ is also essential, and using the spectral theorem 
we can write 
$$
QBQ=C_+-C_-
$$
where $C_\pm\geq 0$ and $C_+C_-=0$, both nonzero.  Choose vectors $\xi_\pm\in C_\pm H$ such 
that 
$$
\langle C_+\xi_+,\xi_+\rangle=\langle C_-\xi_-,\xi_-\rangle>0,   
$$
and let $\rho=\omega_{\xi_+}+\omega_{\xi_-}$.  $\rho$ is a nonzero positive normal 
functional that satisfies $\rho(B)=\rho(QBQ)=0$ and $\rho(A)=0$ because $\rho$ 
lives in $AH^\perp$.  Hence $\rho(S_0)=\{0\}$, and it follows after normalization 
that $\rho$ is a normal state other than $\pi_\infty$ that agrees with $\pi_\infty$ 
on the span of $\{\mathbf 1, V, V^*\}$.  Therefore $\pi_\infty$ is not a boundary 
representation, so by Corollary \ref{ncbCor1}, $\{V\}$ is not a hyperrigid generator 
of $\mathcal K$.  
\end{proof}

Now let $V$ be the standard Volterra operator acting on $L^2[0,1]$, 
\begin{equation}\label{voEq2}
Vf(x)=\int_0^x f(t)\,dt, \qquad f\in L^2[0,1].  
\end{equation}

\begin{lem}\label{voLem1}
The real part of $V$ is $\frac{1}{2}E$, where $E$ 
is the projection on the one dimensional space 
of constant functions.  The imaginary part of $V$ is unitarily equivalent 
to the following diagonal operator $D$ on $\ell^2(\mathbb Z)$: 
\begin{equation}\label{voEq1}
(Du)(n)=\frac{-1}{(2n+1)\pi}u(n), \qquad n\in\mathbb Z, \quad u\in \ell^2(\mathbb Z).  
\end{equation}
In particular, $V$ belongs to the Schatten class $\mathcal L^p$ iff $p>1$.  
\end{lem}

\begin{proof} This result is surely known, and we merely sketch the 
argument.  The adjoint of $V$ is the operator 
$$
V^*f(x)=\int_x^1 f(t)\,dt.  
$$
It follows that $V+V^*$ is the projection on the space of constants, and 
moreover $V^*=E-V$, so that $i\Im V$ is the skew adjoint compact operator 
$$
A=\frac{1}{2}(V-V^*)= V-\frac{1}{2}E.  
$$
To solve the eigenvalue problem for $A$ one sets $Af=\lambda f$ and 
differentiates (in the sense of distributions) to obtain $f=\lambda f^\prime$.  
There are no nonzero solutions $f$ when $\lambda=0$, and for $\lambda\neq 0$ 
we must have $f(x)=C\cdot e^{\omega x}$ for some imaginary $\omega\in\mathbb C$.  Substitution of 
the latter expression for $f$ in the equation $Af=\lambda f$ leads to a solution iff 
$\omega=(2n+1)\pi i$ for some $n\in\mathbb Z$, and the 
possible values of $\lambda$ are 
$$
\lambda_n=\frac{1}{\omega_n}=\frac{1}{(2n+1)\pi i}, \qquad  n\in \mathbb Z,   
$$
with corresponding eigenfunctions $f_n(x)=e^{(2n+1)\pi i x}$, $n\in\mathbb Z$.   
In particular, the asserted form (\ref{voEq1}) for the imaginary part of 
$V$ follows.  
\end{proof}

We conclude: 

\begin{cor}\label{voCor1}  The Volterra operator $V$ of (\ref{voEq2}) satisfies the hypotheses of 
Theorem \ref{voThm1} above, and therefore its conclusion as well.  
\end{cor}

\section{Applications III: Positive maps on matrix algebras}\label{S:cr}

Let  $f:[a,b]\to\mathbb R$ be a continuous function defined on a compact real interval.  Notice 
that if $\phi: \mathcal A\to \mathcal B$ is a unit-preserving positive linear map of unital 
\cstar s and $A$ is a self adjoint operator in $\mathcal A$ with spectrum in $[a,b]$, 
then $\phi(A)$ is an operator in $\mathcal B$ with similar properties.  We will say that 
$f$ is {\em rigid} if for every UCP map of unital \cstar s $\phi: \mathcal A\to \mathcal B$ and every 
self adjoint operator $A\in \mathcal A$ with $\sigma(A)\subseteq [a,b]$, one has 
\begin{equation}\label{crEq1}
\phi(f(A))=f(\phi(A))\implies \phi(A^n)=\phi(A)^n, \quad \forall n=1,2,\dots.  
\end{equation}
The purpose of this section is to identify rigid functions in the following
sense.  We show that rigid functions must be either strictly convex or strictly concave 
in Proposition \ref{crProp3}.  In this commutative context, 
Conjecture \ref{ncbCon1} would imply the converse, namely that every strictly convex function is 
rigid.  While we are unable to prove that assertion, we do prove it for 
operators $A\in\mathcal B(H)$ and maps $\phi: \mathcal B(H)\to\mathcal B(K)$ 
when $\dim H<\infty$ in Theorem \ref{crThm1}.

Fix a real valued function  $f\in C[a,b]$ and 
let $u\in C[a,b]$ be the coordinate function 
$u(x)=x$, $x\in [a,b]$.  In order to determine whether $f$ is rigid, we must first  
identify the Choquet boundary of the function system generated by $f$ 
and the coordinate function $u(x)=x$, $x\in [a,b]$.

\begin{prop}\label{crProp1} Let $S\subseteq C[a,b]$ be the function system spanned by the three functions 
$f, u, \mathbf 1$ and let 
$$
\Gamma=\{(x,f(x): x\in [a,b]\}\subseteq \mathbb R^2 
$$ 
be the graph of $f$ and let $\co\Gamma\subseteq \mathbb R^2$ be its (necessarily compact) convex hull.  
The Choquet boundary of $S$ is the set of points $x\in [a,b]$ such that $(x,f(x))$ 
is an extreme point of $\co\Gamma$.  
\end{prop}

\begin{proof}
The proof, an exercise in elementary convexity theory, is a 
consequence of the following   
three observations.  First, the 
state space of $S$ is naturally identified with the space of all probability measures 
on the compact convex set $K=\co\Gamma$.  Second, the extreme points of $K$ are the points 
$k\in K$ for which the point mass $\delta_k$ is the unique probability measure on $K$ having $k$ as its  
barycenter; and such points must belong to $\Gamma$.  
Third, the Choquet boundary of $S$ is 
identified with the points of $\Gamma$ that have the extremal property of the preceding sentence.  
\end{proof}

Proposition \ref{crProp1} identifies the Choquet boundary $\partial_S[a,b]$ 
 of the function system $S\subseteq C[a,b]$ spanned by 
$f,u,\mathbf 1$.  
If one combines that with the following result, one identifies the 
functions $f$ for which $\partial_S[a,b]=[a,b]$ as precisely those which are 
either strictly convex or strictly concave.

\begin{prop}\label{crProp2}
For every continuous function $f: [a,b]\to \mathbb R$, the following 
are equivalent:
\begin{enumerate}
\item[(i)] Every point of the graph $\Gamma=\{(x,f(x)): x\in [a,b]\}$ of $f$ is 
an extreme point of the convex hull of $\Gamma$.  
\item[(ii)] $f$ is either strictly convex or strictly concave.  
\end{enumerate}
\end{prop}

\begin{proof} (i)$\implies$(ii):  Assuming that (i) holds, we claim first that if $x_1, \dots, x_n$
are points of $I=[a,b]$ 
and $t_1,\dots, t_n$ are positive numbers with sum $1$, then 
\begin{equation}\label{oldcrEq2}
f(t_1x_1+\cdots+t_nx_n)=t_1f(x_1)+\cdots+t_nf(x_n)\implies x_1=\cdots=x_n.  
\end{equation}
Indeed, the left side of the implication implies that for $x_0=t_1x_1+\cdots+t_nx_n$, 
$$
(x_0,f(x_0))=t_1\cdot(x_1,f(x_1))+\cdots+t_n\cdot(x_n,f(x_n))
$$
which by (i) implies $x_1=\cdots=x_n=x_0$.  Next, we claim that for 
any two pairs of distinct points $x\neq y$, $u\neq v$ in $I$ and $0<s,t<1$, 
the two inequalities 
\begin{align*}
f(sx+(1-s)y)&<sf(x)+(1-s)f(y)
\\
f(tu+(1-t)v)&>tf(u)+(1-t)f(v)
\end{align*} 
cannot both hold.  For if they do, then the function $F: [0,1]\to \mathbb R$ 
\begin{align*}
F(\lambda)=&f(\lambda(sx+(1-s)y))+(1-\lambda)(tu+(1-t)v)) 
\\
&-\lambda(sf(x)+(1-s)f(y)) - (1-\lambda)(tf(u)+(1-t)f(v))
\end{align*}
is continuous, positive at $\lambda=0$ and negative at $\lambda=1$, so by the intermediate 
value theorem, there is a $\lambda\in (0,1)$ for which $F(\lambda)=0$, which 
contradicts 
(\ref{oldcrEq2}) for $x_1=x, x_2=y, x_3=u, x_4=v$.  Hence one or the other inequalities 
must be satisfied throughout, so that $f$ is either strictly convex or strictly concave.  
The proof of (ii)$\implies$(i) is straightforward.  
\end{proof}

\begin{prop}\label{crProp3}
A rigid function $f\in C[a,b]$ is either strictly convex or strictly concave.  
\end{prop}

\begin{proof}[Proof of Proposition \ref{crProp3}] 
 We actually prove a somewhat stronger version of Proposition \ref{crProp3} 
in its contrapositive formulation: {\em  
If $f$ is neither strictly convex nor strictly concave, then there is a finite dimensional 
Hilbert space $H$, a self adjoint operator $A\in\mathcal B(H)$ with spectrum in $[a,b]$, and a 
unital completely positive map 
$\phi: \mathcal B(H)\to\mathcal B(H)$ such that $\phi(A)=A$, $\phi(f(A))=f(\phi(A))$, 
but $\phi$ is not multiplicative on the algebra of polynomials in $A$.}  In particular, 
$f$ is not rigid.

Indeed, let $f:[a,b]\to \mathbb R$ be a continuous function that 
is neither strictly convex nor strictly concave.  By Proposition \ref{crProp2}, 
the graph $\Gamma$ of $f$ must contain some point $(x_0,f(x_0))$ that is not an extreme point 
of its convex hull, and hence can be written as a nontrivial convex combination 
of two distinct points of the convex hull of $\Gamma$.  Since $\Gamma\subseteq \mathbb R^2$, 
every point of the convex 
hull of $\Gamma$ is a convex combination of at most $3$ points of $\Gamma$, 
and we conclude that $(x_0, f(x_0))$ 
can be written as a convex combination of at most $6$ points of $\Gamma$, 
$(x_i, f(x_i))$, $x_i\neq x_0$, $i=1,\dots, n\leq 6$.  
By discarding some of the points $x_1,\dots, x_n$ and 
reducing $n$ if necessary, 
we can assume that the $n+1$ points $x_0, x_1, \dots, x_n\in [a,b]$ are distinct.  
By the choice of $x_1, \dots, x_n$, there are numbers $t_1, \dots, t_n\in[0,1]$ such that 
\begin{equation}\label{crEq7}
x_0=\sum_{k=1}^n t_kx_k , \quad {\rm{and }}\quad  f(x_0)=\sum_{k=1}^nt_k\cdot f(x_k),   
\end{equation}
and consider the positive linear map $\phi:\mathbb C^{n+1}\to\mathbb C^{n+1}$ defined by 
$$
\phi(\lambda_0, \lambda_1,\dots, \lambda_n)=(t_1\lambda_1+\cdots+t_n\lambda_n, \lambda_1, \dots, \lambda_n), 
\qquad \lambda_k\in \mathbb C.  
$$
Viewing $\mathbb C^{n+1}$ as the algebra of all diagonal matrices in $M_{n+1}=M_{n+1}(\mathbb C)$, 
$\phi$ becomes a unit-preserving positive (and hence completely positive) linear map.  
We may extend $\phi$ to a UCP map  $\tilde\phi: M_{n+1}\to M_{n+1}$ in many ways, for example, by composing 
it with the trace-preserving conditional expectation of $M_{n+1}$ 
onto the diaganal subalgebra.  In order to conserve notation, we 
continue to write $\phi$ for an extension of the original 
map on diagonal matrices to a completely positive map of $M_{n+1}$ into itself.  

Consider the diagonal operator 
$$
A=(x_0, x_1,\dots, x_n).  
$$
The conditions (\ref{crEq7}) on $t_1,\dots,t_n$ imply the two operator formulas 
$$
\phi(A)=A, \qquad \phi(f(A))=f(A).  
$$
Finally, since $x_i\neq x_j$ for $i\neq j$, the algebra generated by $A$ is all diagonal 
sequences in $\mathbb C^{n+1}$, and obviously $\phi$ does not 
fix all diagonal sequences.  Together, these properties imply that 
the restriction of $\phi$ to the algebra of 
polynomials in $A$ is not multiplicative.   
\end{proof}

\begin{thm}\label{crThm1} Let $[a,b]$ be a compact real interval and 
let $f\in C[a,b]$ be real valued.  If $f$ is strictly convex, then for every 
pair  $H$, $K$  of finite dimensional 
Hilbert spaces, every self adjoint operator $A\in\mathcal B(H)$ having 
spectrum in $[a,b]$,  and every UCP map $\phi: \mathcal B(H)\to \mathcal B(K)$ satisfying 
\begin{equation}\label{oldcrEq1}
\phi(f(A))=f(\phi(A)),   
\end{equation}
the restriction of $\phi$ to the algebra of polynomials in $A$ 
is multiplicative.  

Conversely, if neither $f$ nor $-f$ is strictly convex, then 
there is a Hilbert space $H$ of dimension at most $7$, a self adjoint operator 
$A\in\mathcal B(H)$ with spectrum in $[a,b]$, and a UCP map $\phi: \mathcal B(H)\to \mathcal B(H)$ 
such that 
$$
\phi(A)=A, \quad \phi(f(A))=f(A), 
$$
and which is not multiplicative on the algebra of polynomials in $A$.  
\end{thm}

\begin{proof}  To prove the first paragraph, let 
$\phi: \mathcal B(H)\to \mathcal B(K)$ be a UCP map, let $A=A^*\in\mathcal B(H)$ have its  
spectrum in $I$ and satisfy 
(\ref{oldcrEq2}), and assume that $f$ is strictly convex.  Let $B=\phi(A)\in \mathcal B(K)$.  
Consider the representations 
$\pi: C[a,b]\to \mathcal B(H)$ and  $\sigma: C[a,b]\to \mathcal B(K)$ defined by 
$$
\pi(g)=g(A), \quad \sigma(g)=g(B),  \qquad g\in C(X).    
$$
Let $u(x)=x$, $x\in X$, be the coordinate function and let $S$ be the $2$ or $3$ dimensional 
function system spanned by $u, f$ and the constants.   We have arranged that 
$\phi(\pi(u))=\sigma(u)$, and $\phi(\pi(f))=\sigma(f)$, hence  
\begin{equation}\label{oldcrEq3}
\phi\circ\pi\restriction_{S}=\sigma\restriction_{S}.  
\end{equation}

Since $K$ is finite dimensional, $\sigma$ is a finite direct sum of (one dimensional) irreducible 
representations of $C[a,b]$, and such representations correspond to points of $[a,b]$.  Since $f$ is 
assumed to be strictly convex, Proposition 
\ref{crProp2} implies that every point of the graph $\Gamma$ of $f$ is an extreme pont of the 
convex hull of $\Gamma$; and Proposition \ref{crProp1} implies that every point of $[a,b]$ 
belongs to the Choquet boundary of $[a,b]$ relative to $S$.  Hence $\sigma$ is a direct 
sum of one dimensional representations with the unique extension property.  By Proposition 
\ref{ncbProp1}, $\sigma$ itself has the unique extension property; and since 
$\phi\circ\pi$ restricts to $\sigma$ on $S$, it follows that $\phi\circ\pi=\sigma$.  
Hence the restriction of $\phi$ to $\pi(C[a,b]))$ is multiplicative.  

The assertion of the second paragraph follows from the proof of Proposition \ref{crProp2}.  Indeed, 
the construction in that proof exhibits a Hilbert space $H$ of dimension at most 7, an operator 
$A=A^*\in\mathcal B(H)$ and a unital completely positive map $\phi: \mathcal B(H)\to \mathcal B(H)$ 
with the stated properties.  
\end{proof}

\begin{rem}[Infinite dimensional generalizations]\label{crRem1} Naturally, one would hope that 
the second paragraph of Theorem \ref{crThm1} remains true if one drops the 
hypothesis of finite dimensionality of $H$; but 
that has not been proved.  Note that it would be enough to prove 
Conjecture \ref{ncbCon1} for commutative \cstar s.  In turn, that would 
provide a generalization of Theorem \ref{saThm1} 
to cases in which $G=\{\mathbf 1, x,x^2\}$ 
is replaced with $G=\{\mathbf 1, x, f(x)\}$ for any continuous 
strictly convex function $f$ and any self adjoint operator $x$.    
\end{rem}

\part{A local version of Conjecture \ref{ncbCon1} }\label{P:local}

It is conceivable that Conjecture 
\ref{ncbCon1} might fail for reasons yet unknown; and in that event one needs to know what {\em can} be 
proved.  
In the remaining sections we take up this issue in the commutative case 
of function systems $S\subseteq C(X)$, where $X$ is a compact metric space, and 
we show that function systems satisfy a ``localized" version 
of Conjecture \ref{ncbCon1}.

More precisely,  let 
$S\subseteq C(X)$ be a linear space of continuous functions that separates 
points, contains the constants, is closed under complex conjugation, and assume that every 
point $p\in X$ has a unique representing measure in the sense that the only probability 
measure $\mu$ on $X$ satisfying 
$$
f(p)=\int_X f\,d\mu, \qquad f\in S  
$$
is the point mass $\mu=\delta_p$.
By Theorem \ref{mrThm1}, to prove Conjecture \ref{ncbCon1} 
it is enough to prove the following assertion:  For every separably-acting representation 
$\pi: C(X)\to\mathcal B(H)$ and every positive 
linear map $\phi: C(X)\to \mathcal B(H)$ such that $\phi\restriction_S=\pi\restriction_S$, 
one has 
\begin{equation}\label{PlocalEq1}
\phi(f) = \pi(f),\qquad f\in C(X).  
\end{equation}
Let  $E$ be the spectral measure of $\pi$ --  namely the projection valued measure 
on the $\sigma$-algebra of Borel subsets of $X$ that satisfies  
$$
\pi(f)=\int_X f(x)\,dE(x),\qquad f\in C(X).  
$$
We 
will show that (\ref{PlocalEq1}) is true {\em locally} in the following sense:   
For every 
positive linear map $\phi: C(X)\to \mathcal B(H)$ that restricts 
to $\pi$ on $S$ and for every point $p\in X$,  
\begin{equation}\label{pLocalEq2}
\lim_{\epsilon\to 0}\|(\phi(f)-\pi(f))E(B_\epsilon(p))\|=0, \qquad f\in C(X),   
\end{equation}
where $B_\epsilon(p)=\{x\in X: d(x,p)\leq \epsilon\}$ is the ball of radius $\epsilon>0$ 
about $p$.   Indeed, the limit (\ref{pLocalEq2}) is zero uniformly in $p$ (see  Theorem \ref{ueThm1}).  

\section{The local \cstar\ of a representation of $C(X)$}\label{S:lo}

Throughout this section, $X$ will denote a compact metric space with 
metric $d:X\times X\to [0,\infty)$.    Every 
representation 
$\pi: C(X)\to\mathcal B(H)$ gives rise to a spectral measure $F\to E(F)$ 
on the Borel subsets $F\subseteq X$, and which is uniquely defined 
by
$$
\langle\pi(f)\xi,\xi\rangle =\int_X f(x)\langle E(dx)\xi,\xi\rangle, \qquad \xi\in H, \quad f\in C(X).  
$$
We say that $\pi:C(X)\to\mathcal B(H)$ is a {\em separable} representation if the space $H$ on which it acts 
is a separable Hilbert space.  
All representations $\pi$ are assumed to be nondegenerate, so that $\pi(\mathbf 1)=\mathbf 1$.

Let $\pi: C(X)\to\mathcal B(H)$ be a representation and let $p\in X$.  An 
operator $A\in\mathcal B(H)$ is said to be {\em locally null at $p$} if for every 
$\epsilon>0$ there is an open neighborhood $U$ of $p$ such that 
$\|AE(U)\|\leq \epsilon$ and $\|A^*E(U)\|\leq \epsilon$.

\begin{prop}\label{loProp1}
Let $\pi: C(X)\to \mathcal B(H)$ be a representation.  Then for every operator 
$A\in\mathcal B(H)$ the following are equivalent:
\begin{enumerate}
\item[(i)] $A$ is locally null at every point of $X$.  
\item[(ii)] $A$ is uniformly locally null in the following sense: Letting 
$B_\delta(p)=\{q\in X: d(p,q)<\epsilon\}$ be the $\delta$-ball about a point $p\in X$, we have 
\begin{equation}\label{loEq1}
\sup_{p\in X}(\|AE(B_\delta(p))\|+\|A^*E(B_\delta(p))\|)\to 0 \quad {\rm{as\ }} \delta\to 0+.  
\end{equation}
\end{enumerate}
\end{prop}

\begin{proof}(i) $\implies$ (ii):  It suffices to show that for every operator 
$A\in\mathcal B(H)$, 
\begin{equation}\label{loEq1.1}
\lim_{\delta\to0}\|AE(B_\delta(p))\|=0\ \forall p\in X\implies 
\lim_{\delta\to0} \sup_{p\in X}\|AE(B_\delta(p))\|=0.  
\end{equation}
Contrapositively, 
let $\delta_n>0$ be a sequence tending to $0$ such that 
\begin{equation}\label{loEq2}
\|A E(B_\delta(p_n))\|\geq \alpha>0, \qquad n=1,2,\dots.  
\end{equation}
By compactness, $\{p_n\}$ has a convergent subsequence,  and by passing to 
that subsequence we may assume that $p_n\to p\in X$ as $n\to \infty$.  For every 
$\delta>0$ we will have $B_\delta(p)\supseteq B_{\delta_n}(p_n)$ for sufficiently 
large $n$, and for such $n$, (\ref{loEq2}) implies 
$$
\|AE(B_\delta(p))\|\geq \|A E(B_\delta(p_n))\|
\geq \alpha, 
$$
from wich we conclude 
$$
\inf_{\delta>0}\|A E(B_\delta(p))\|\geq \alpha, 
$$
contradicting item (i) at the point $p$.   
(ii) $\implies$ (i) is trivial.  
\end{proof}

\begin{defn}
Let $\pi: C(X)\to \mathcal B(H)$ be a representation.  An operator $A\in\mathcal B(H)$ 
is said to be {\em locally null} (relative to $\pi$) if it satisfies the equivalent conditions of 
Proposition \ref{loProp1}.  $\mathcal N_\pi$ will denote the set of all operators 
that are locally null with respect to $\pi$.  
\end{defn}

\begin{rem}[Structure of $\mathcal N_\pi$]\label{loRem1}
Consider the linear space of operators 
$$
\mathcal L_\pi=\{A\in\mathcal B(H): \lim_{\delta\to0}\|AE(B_\delta(p))\|=0, \quad \forall p\in X\}.  
$$
Obviously, $\mathcal L_\pi$ is a norm-closed left ideal in $\mathcal B(H)$ for which 
$\mathcal N_\pi=\mathcal L_\pi\cap\mathcal L_\pi^*$.  
Moreover, the norm-closed linear span of $\mathcal L_\pi\cdot\mathcal L_\pi^*$ 
is a two-sided ideal in $\mathcal B(H)$,  which when nonzero can only be the $C^*$-algebra 
$\mathcal K$ of all compact operators on $H$ or all of 
$\mathcal B(H)$.  
We conclude 
that {\em either a) $\mathcal N_\pi=\{0\}$,  or b) $\mathcal N_\pi=\mathcal K$, 
or c) $\mathcal N_\pi$ contains $\mathcal K$ together with some 
noncompact operators, in which case it is strongly Morita equivalent to $\mathcal B(H)$. } 
\end{rem}

\begin{prop}\label{loProp2}    If $\pi: C(X)\to\mathcal B(H)$ is a separable representation with  
no point spectrum, then $\mathcal N_\pi$ contains the \cstar\ $\mathcal K$ of 
compact operators.  
\end{prop}

\begin{proof} We claim 
first that $\mathcal N_\pi$ contains every rank one projection $A\in\mathcal K$.  Indeed, let 
$A\xi=\langle\xi,f\rangle f$, where $f$ is a unit 
vector in $H$.  
Then for every $p\in X$ and $\delta>0$, $AE(B_\delta(p))$ is a rank one operator with 
$$
\|AE(B_\delta(p))\|^2=\|E(B_\delta(p))f\|^2=\langle E(B_\delta(p))f,f\rangle, 
$$
and the latter tends to zero as $\delta\downarrow 0$ because the hypothesis 
on $\pi$ implies that the probability measure 
defined on $X$ by $\mu(S)=\langle E(S)f,f\rangle$ is 
nonatomic.  Hence $A\in\mathcal N_\pi$.  The spectral theorem implies 
every self adjoint compact operator can be norm approximated 
by linear combinations of rank one projections, 
hence $\mathcal N_\pi\supseteq \mathcal K$.  
\end{proof}

The basic facts that connect $\mathcal N_\pi$ to the structure of $X$ are as follows:  

\begin{prop}\label{loProp3}
If $X$ is countable then $\mathcal N_\pi=\{0\}$ for every separable 
representation $\pi: C(X)\to \mathcal B(H)$.  
If $X$ is uncountable, then there is a separable representation $\pi$ of $C(X)$ 
such that $\mathcal N_\pi$ contains {\em non}-compact operators, and in 
fact $\mathcal N_\pi$ is strongly Morita equivalent to $\mathcal B(H)$.  
\end{prop}

\begin{proof} Assume that $X$ is countable and let $\pi: C(X)\to\mathcal B(H)$ be 
a separably acting representation.  The set of factor representations 
of $C(X)$  being countable ($\cong X$), 
reduction theory shows that $\pi$ decomposes into a direct sum of disjoint factor representations, which 
in this simple context means 
$$
\pi(f)=\sum_{n\geq 1}^\oplus f(p_n)E_n
$$
where the $E_k$ are a sequence of mutually orthogonal projections with sum $\mathbf 1$ 
and $p_1,p_2,\dots$ is a (finite or infinite) sequence of distinct points of $X$.  Hence 
the 
spectral measure of $\pi$ is atomic and is concentrated on $\{p_1,p_2,\dots\}$.  
It follows that for every operator $A\in \mathcal N_\pi$ we must have 
$$
\|AE_n\|=\inf_{\delta>0}\|AE(B_\delta(p_n))\|=0, \qquad n=1,2,\dots, 
$$
hence $A=\sum_n AE_n=0$.  

Assume now that $X$ is uncountable.  Since $X$ can be viewed a standard Borel space, it 
contains a Borel subset that is isomorphic to the unit interval $[0,1]$, and 
hence $X$ supports a nonatomic Borel probability measure $\mu$.  
Let $H=L^2(X,\mu)$ and let $\pi$ be the usual representation 
of $C(X)$ on $L^2(X,\mu)$ in which $\pi(f)$ acts as multiplication by $f$.  

By Proposition \ref{loProp2}, $\mathcal N_\pi$ contains all compact operators on $H$.  
Now let $\infty\cdot\pi$ be the direct sum of a countably infinite 
number of copies of $\pi$.  For 
every compact operator $K\in \mathcal B(H)$, the direct sum 
$\infty\cdot K=K\oplus K\oplus\cdots$ of copies of $K$ must belong to 
$\mathcal N_{\infty\cdot\pi}$.  Since none of the operators $\infty\cdot K$ is compact 
when $K\neq 0$, Remark \ref{loRem1} implies that 
$\infty\cdot\pi$ is a representation of $C(X)$ with the stated properties.  
\end{proof}

\section{Local uniqueness of UCP extensions}\label{S:ue}

Continuing our discussion of function systems $S\subseteq C(X)$ on compact 
metric spaces $X$, in this section we prove:

\begin{thm}\label{ueThm1}  Given a separable representation $\pi: C(X)\to\mathcal B(H)$, 
let $\phi: C(X)\to \mathcal B(H)$ 
be a UCP map such that $\phi(s)-\pi(s)\in \mathcal N_\pi$ for every $s\in S$.  If every 
point of $X$ belongs to the Choquet boundary $\partial_SX$, then 
$$
\phi(f)-\pi(f)\in\mathcal N_\pi, \qquad \forall \ f\in C(X).  
$$ 
\end{thm}

The proof of Theorem \ref{ueThm1} requires the following estimate:

\begin{prop}\label{ueProp2} 
Let $S\subseteq C(X)$ be an arbitrary function system and let 
$\phi: C(X)\to \mathcal B(H)$ be a UCP map 
with the property 
$$
\phi(g)-\pi(g)\in\mathcal N_\pi,\qquad \forall g\in S.  
$$ 
Then for $p\in X$ and every $f\in C(X)$ 
we have 
\begin{equation}\label{ueEq1}
\limsup_{n\to\infty}\|\phi(f)E(B_{1/n}(p))\|^2\leq 
\inf\{s(p): s\in S, \ s\geq |f|^2\}. 
\end{equation}
\end{prop}  

\begin{proof} For each $n=1,2,\dots$ choose a unit vector $\xi_n\in E(B_n(p))H$ such 
that 
\begin{equation}\label{ueEq2}
\|\phi(f)E(B_{1/n}(p))\|^2\leq \|\phi(f)\xi_n\|^2+\frac{1}{n},   
\end{equation}
and fix a function $s\in S$ satisfying 
$s\geq |f|^2$.  Then   
\begin{equation*}
\|\phi(f)\xi_n\|^2=\langle\phi(f)^*\phi(f)\xi_n,\xi_n\rangle\leq \langle\phi(|f|^2)\xi_n,\xi_n\rangle
\leq\langle\phi(s)\xi_n,\xi_n\rangle.  
\end{equation*}
Now fix $\epsilon>0$.  Since $\xi_n$ is a unit vector in  $E(B_{1/n}(p))H$, 
it follows from the hypothesis $\phi(s)-\pi(s)\in \mathcal N_\pi$ that 
for sufficiently large $n$ we will have 
$$
|\langle (\phi(s)-\pi(s))\xi_n,\xi_n\rangle|\leq \|(\phi(f)-\pi(f))E(B_{1/n}(p))\|\leq \epsilon
$$
and therefore 
$$
\|\phi(f)\xi_n\|^2\leq\langle\phi(s)\xi_n,\xi_n\rangle \leq \langle \pi(s)\xi_n,\xi_n\rangle+\epsilon
= \int_X s(x)\langle E(dx)\xi_n,\xi_n\rangle+\epsilon.  
$$
Since $\xi_n\in B_{1/n}(p))$, the measure $\langle E(\cdot)\xi_n,\xi_n\rangle$ is 
supported on the closure of $B_{1/n}(p)$.  Hence the term on the right is dominated by 
$$
\sup\{s(x): d(x,p)\leq 1/n\}+\epsilon 
$$ 
which, by continuity of $s$ at $p$, is in turn dominated by 
$s(p)+2\epsilon$ for sufficiently large $n$.  Finally, since $\epsilon$ 
can be arbitrarily small, we obtain
\begin{equation*}
\limsup_{n\to\infty}\|\phi(f)\xi_n\|^2\leq s(p).  
\end{equation*}
From (\ref{ueEq2}) we conclude that 
$$
\limsup_{n\to\infty}\|\phi(f)E_n\|^2\leq s(p),  
$$
and the estimate (\ref{ueEq1}) follows after taking the infimum over $s$.  
\end{proof}

We will also make use of  the following property of points with unique representing measures, a 
consequence 
of a more general 
minimax principle based on the Hahn-Banach theorem (see formula (1.2) of \cite{GlickFandM}):

\begin{lem}\label{ueLem1}
Let $u$ be a real function in $C(X)$ and let 
$p$ be a point in the Choquet boundary of $X$ relative to $S$.  Then 
$$
u(p)=\inf\{s(p): s\in S, \ s\geq u\}.  
$$
\end{lem}

\begin{proof}[Proof of Theorem \ref{ueThm1}]  Fix $f\in C(X)$ and let $A=\phi(f)-\pi(f)$.  We have to show that for every 
point $p\in X$ 
\begin{equation}\label{ueEq4} 
\lim_{n\to\infty}\|AE(B_{1/n}(p))\|=0.  
\end{equation}
Fixing $p$, note that by replacing $f$ with $f-f(p)\mathbf 1$, it suffices to prove 
(\ref{ueEq4}) for functions $f$ that vanish at $p$.   For such a function $f$ we claim first that  
\begin{equation}\label{ueEq3}
\lim_{n\to\infty}\|\pi(f)E(B_{1/n}(p))\|=0
\end{equation}
Indeed, we have 
$$
\|\pi(f)E(B_{1/n}(p))\|=\|\int_{B_{1/n}(p)}f(x)\,E(dx)\|\leq \sup_{x\in B_{1/n}(p)}|f(x)|,
$$
and the term on the right tends to $|f(p)|=0$ as $n\to\infty$.  

So to prove (\ref{ueEq4}), we have to show that $\|\phi(f)E(B_{1/n}((p))\|$ tends to 
zero as $n\to\infty$.  To see that, note that Proposition 
\ref{ueProp2} implies 
\begin{equation}\label{ueEq5}
\limsup_{n\to\infty}\|\phi(f)E(B_{1/n}(p))\|\leq \inf\{s(p): s\in S,\ s\geq |f|^2\}.  
\end{equation}
Since $p$ belongs to the Choquet boundary,  Lemma \ref{ueLem1} implies 
that the right side of (\ref{ueEq5}) is $|f(p)|^2=0$.  Thus  
(\ref{ueEq4}) is proved.  
\end{proof}

\bibliographystyle{alpha}

\newcommand{\noopsort}[1]{} \newcommand{\printfirst}[2]{#1}
  \newcommand{\singleletter}[1]{#1} \newcommand{\switchargs}[2]{#2#1}

\end{document}